\documentclass[12pt,twoside, a4paper]{amsart}

\usepackage[x11names]{xcolor}
\usepackage[linkcolor=Blue2,citecolor=red,colorlinks]{hyperref}
\usepackage[capitalize,nameinlink]{cleveref}
\usepackage[linkcolor=Blue2,citecolor=red,colorlinks]{hyperref}
\usepackage{color, bm, amscd, tikz-cd, amssymb}
\usepackage{color, bm, amscd, tikz-cd}

\usepackage{comment}
\usepackage{centernot}
\usepackage{mathtools}
\usepackage{stmaryrd}
\makeatletter
\newcommand{\xMapsto}[2][]{\ext@arrow 0599{\Mapstofill@}{#1}{#2}}
\def\Mapstofill@{\arrowfill@{\Mapstochar\Relbar}\Relbar\Rightarrow}
\makeatother

\newcommand{\cA}{{\mathcal A}}
\newcommand{\cB}{{\mathcal B}}

\newcommand{\cD}{{\mathcal D}}
\newcommand{\cE}{{\mathcal E}}

\newcommand{\cH}{{\mathcal H}}
\newcommand{\cI}{{\mathcal I}}

\newcommand{\cK}{{\mathcal K}}

\newcommand{\cQ}{{\mathcal Q}}

\newcommand{\cS}{{\mathcal S}}

\newcommand{\cX}{{\mathcal X}}

\newcommand{\cZ}{{\mathcal Z}}

\newcommand{\bT}{{\mathbb{T}}}

\newcommand{\bD}{{\mathbb D}}
\newcommand{\bC}{{\mathbb C}}

\newtheorem{thm}{Theorem}[section]

\newtheorem{lemma}[thm]{Lemma}

\newtheorem{proposition}[thm]{Proposition}

\theoremstyle{definition}

\newtheorem{definition}[thm]{Definition}
\newtheorem{remark}[thm]{Remark}
\newtheorem{problem}[thm]{Problem}

\newtheorem{example}[thm]{Example}
\numberwithin{equation}{section}

\begin{document}
\title[Joint Spectrum]{Cayley--Hamilton tuples: an interplay between algebraic varieties and joint spectra}

\author[Das]{B. Krishna Das}
\address[Das]{Department of Mathematics, Indian Institute of Technology Bombay, Powai, Mumbai, 400076, India}
\email{dasb@math.iitb.ac.in, bata436@gmail.com}

\author[Kumar]{Poornendu Kumar}
\address[Kumar]{Department of Mathematics, Indian Institute of Science Bangalore, Karnataka, 560012, India. }
\email{poornendukumar@gmail.com}

\author[Sau]{Haripada Sau}
\address[Sau]{Department of Mathematics, Indian Institute of Science Education and Research, Dr.\ Homi Bhabha Road, Pashan, Pune, Maharashtra 411008, India.}
\email{haripadasau215@gmail.com, hsau@iiserpune.ac.in}

\subjclass[2020]{14M12, 47A13}
\keywords{Algebraic varieties, Cayley--Hamilton tuples, Distinguished Varieties, Spectra, Annihilating ideals, Polydisk}

\begin{abstract}
We introduce the notion of \textit{Cayley--Hamilton tuples}: these are commuting operator tuples that are annihilated by a non-zero polynomial and such that its Taylor joint spectrum coincides with the algebraic variety determined by its annihilating ideal. Commuting matrix tuples are Cayley--Hamilton tuples. We provide two families of Cayley--Hamilton tuples in the infinite dimensional setting with additional details. What arises as a by-product is a concrete characterization of distinguished varieties in the polydisk in terms of Taylor joint spectrum of commuting isometries. These varieties have been of interest in various fields of mathematics over the last two decades. The Taylor and Waelbroeck joint spectrum of a Cayley--Hamilton tuple are shown to be the same. It is also shown that the support of the annihilating ideal of a Cayley--Hamilton tuple is the same as its joint spectrum. As an application, we deduce an algebraic characterization of bi-variate polynomials whose zero set intersected with the closed bidisk is the joint spectrum of a commuting isometric pair.
\end{abstract}

\maketitle

\section{Background and the main results}
\subsection{Background}A fundamental fact in linear algebra is the Cayley--Hamilton Theorem: \textit{every square matrix $T$ is annihilated by the polynomial $p(z)=\det(T-zI)$ and that $\sigma(T)=\cZ(p)$, the zero set of $p$.} For operators acting on an infinite dimensional Hilbert space, such a result cannot be expected because the zero set of a single-variable polynomial is discrete while the spectrum, in general, is not. There is, however, a class of Hilbert space operators for which a Cayley--Hamilton type result holds. This is the so-called $C_0$ class consisting of those completely non-unitary (cnu) contractive operators $T$ for which there is a bounded analytic function $f$ on the disk $\bD$ such that $f(T)=0$. Sz.-Nagy's dilation theorem plays a pivotal role here ensuring that $f(T)$ is well-defined. Let us denote by $H^\infty(\bD)$ the algebra of bounded analytic functions on $\bD$. For the next component of the Cayley-Hamilton Theorem for $T\in C_0$, one considers the following ideal in $H^\infty(\bD)$ 
$$
\operatorname{Ann}_{H^\infty}(T) :=\{f\in H^\infty(\bD): f(T)=0\}.
$$The non-triviality of the ideal coupled with the celebrated Beurling Theorem gives the following
\begin{align}\label{Ann}
\operatorname{Ann}_{H^\infty}(T) = \varphi H^\infty(\mathbb{D})
\end{align} for some inner function $\varphi$ on $\bD$. The inner function $\varphi$ is often referred to as the \textit{minimal function} for $T$ because if $f(T)=0$ for some $f\in H^{\infty}(\bD)$, then $\varphi$ must divide $f$. Note that the minimal function is unique up to multiplication by a unimodular constant. There is a gratifying converse to this.

 For $d\geq 1$ and a Hilbert space $\cE$, we shall denote by $H^2_{\bD^d}(\cE)$ the Hardy space of $\cE$-valued analytic functions so that the coefficients of its Taylor series expansion around the origin is square summable. This is with the exception that we shall have $H^2_{\bD^d}$ instead of $H^2_{\bD^d}(\bC)$, $H^2(\cE)$ instead of $H^2_{\bD}(\cE)$, and $H^2$ instead of $H^2_{\bD}(\bC)$. The converse is that for a non-constant inner function $\varphi$, the operator defined on $\cK_\varphi:=H^2\ominus \varphi H^2$ by
\begin{align}\label{Compress}
T := P_{\mathcal K_\varphi} M_z \big|_{\mathcal K_\varphi}
\end{align}is in $C_0$ with $\varphi$ as its minimal function.  

Moreover, for an inner function $\varphi$, let $\mbox{supp}(\varphi)$ (read support of $\varphi$) denote the union of the set of points in $\bD$ where $\varphi$ vanishes and the complement in $\bT$ of the arcs of $\bT$ where $\varphi$ has an analytic extension. It is remarkable that if $\varphi_T$ is the minimal function for $T$ in $C_0$, then one obtains a result analogous to the second component of the Cayley--Hamilton Theorem as follows:
\begin{align}\label{TheSupp}
\sigma(T)= \mbox{supp}(\varphi_T).
\end{align}We refer the readers to the books \cite{Hari} and \cite[Sections III.4 and III.5]{Nagy-Foias} for proofs of the results stated here and more in this direction.

\subsection{The problem of interest}The purpose of this paper is to  establish infinite-dimensional, multivariate analogues of the classical results stated above. The first obstruction one faces is the absence of an $H^\infty$-functional calculus for tuples of commuting contractions. Such absence is owing to the fact that the Sz.-Nagy's dilation theorem does not extend to higher variables.  Sz.-Nagy's dilation theorem does extend to two variables (thanks to And\^o \cite{Ando}) and in this case an $H^\infty$-functional calculus is known for commuting pairs of cnu contractions; see the papers \cite{BDO} and \cite{Cooper}. However, the theory could not be extended to the full glory of the Sz.-Nagy--Foias theory for various reasons; the absence of an analogue of the Beurling Theorem for $H^2_{\bD^2}$ is one. Considering all these roadblocks, we resort to operators that are annihilated by polynomials. This prompts the following definition.
\begin{definition}\label{D:algebraic}
An operator tuple $\bm T$ is said to be \textit{algebraic} if there is a non-zero polynomial $p$ such that $p(\bm T)=0$.
\end{definition}

In higher variables, there is a choice of the notion of spectrum. We choose to work with the Taylor spectrum, which is briefly discussed in \S \ref{S:Prelim}. The choice of Taylor spectrum is inspired by its rich properties discovered by Taylor in \cite{Taylor, Taylor2}. We shall also have use of the Waelbroeck spectrum defined originally in \cite{W}.
A good source for a comprehensive discussion of various notions of joint spectrum is the survey by Curto \cite{Cur}.
It is well-known that for commuting matrices $\bm T$, the joint spectrum constitutes of the joint eigenvalues. 
For tuples of operators acting on an infinite dimensional Hilbert space, the set of joint eigenvalues can in general be empty and the joint spectrum remains an obscure object other than that it is always nonempty and compact. 
As a simple illustration of its inherent complexity, recall that as a consequence of the Wold decomposition \cite{Wold}, the spectrum of a single isometry is either a subset of the unit circle $\mathbb{T}$ or the entire closed unit disk $\overline{\bD}$. In sharp contrast, the Taylor joint spectrum of even a pair of commuting isometries is sufficiently intricate to merit independent study; see for example the paper \cite{BRV} where the Taylor joint spectrum is computed for very special classes of commuting isometric pairs. Indeed, the problem of computing the joint spectrum is notoriously difficult, and despite sustained interest, it remains far from being fully understood. To mention a few representative works, we refer to \cite{GimenezPAMS2001, GimenezThesis} for a computation of the Taylor joint spectrum of composition operators, to \cite{BP} that characterizes the Taylor spectrum of pairs of isometries via diagrammatic methods, to \cite{ABLS} that determines the spectrum for pairs of Clark unitary operators, and to the papers \cite{CT}, \cite{DEHS} that studied the spectrum of commuting tuples associated with kernels satisfying the complete Nevanlinna–Pick property on the open unit ball.

Let us also mention that various notions of joint spectrum have been introduced with the aim of establishing connections with algebraic varieties in projective spaces, as well as in affine spaces -- see \cite{ST, SYZ, Yang}. See also \cite{KS}, where a new notion of eigenvalues, depending on the operator space structure of the ambient domain, is introduced to study the zero sets of holomorphic functions in the Schur--Agler class over the unit polydisk, as well as those in the unit ball of the multiplier algebra of the Drury--Arveson space.

The aim of this paper is not to just compute the joint spectrum but rather it revolves around the following problem inspired by the classical Cayley--Hamilton Theorem. 

\begin{problem}\label{P:TheProb} For which commuting contractive operator tuple $\bm T$ does there exist a collection $\cI$ of nonzero polynomials in $\bC[\bm z]$ so that
\begin{enumerate}
\item[(i)] $p(\bm T)=0$ for each $p\in\cI$; and
\item[(ii)] the Taylor joint spectrum of $\bm T$ is given by $\cZ(\cI)\cap\overline{\bD}^d$?
\end{enumerate}
\end{problem}
The Taylor joint spectrum of a commuting operator tuple $\bm T$ will throughout be denoted by $\sigma_T(\bm T)$.
As a consequence of the polynomial mapping theorem, it follows that condition (i) in Problem \ref{P:TheProb} implies a weaker version of (ii), viz., the Taylor joint spectrum of $\bm T$ must be contained in $\cZ(\cI)\cap\overline{\bD}^d$. Thus, condition (ii) ensures that the size of the collection $\cI$ is big enough for this containment to be an equality. For an algebraic tuple, a natural choice for $\cI$ is the annihilating ideal in $\bC[\bm z]$:
$$
\operatorname{Ann}(\bm T)=\{p\in\bC[\bm z]: p(\bm T)=0\}.
$$
It follows from the polynomial mapping theorem again that if Problem \ref{P:TheProb} is solved for $\bm T$ and any collection $\cI$ of polynomials, then condition (ii) (and (i) of course) holds for $\operatorname{Ann}(\bm T)$ as well. For this reason, we shall be concerned with finding a collection $\cI$ of annihilating polynomials so that (i) and (ii) hold instead of finding the annihilating ideal for an algebraic tuple $\bm T$. Let us introduce the following terminology.
\begin{definition}\label{D:C-H}
We say that a commuting tuple $\bm T$ of Hilbert space operators is a \textit{Cayley--Hamilton tuple} if Problem \ref{P:TheProb} can be solved for it in the affirmative.
\end{definition}It is easy to see that the class of Cayley--Hamilton tuples is closed under adjoint, direct sum and unitary equivalence. We shall provide three classes of Cayley--Hamilton tuples in this paper. We proceed with a number of concrete non-examples for further illustration. We shall have use of the following terminology.
\begin{definition}\label{D:AlgVar}
We say that a subset of the closed polydisk $\overline{\bD}^d$ is an \textit{algebraic variety} if it coincides with the common zero set of a collection of non-zero polynomials in $d$ variables within $\overline{\bD}^d$.
\end{definition}
Thus, condition (ii) of Problem \ref{P:TheProb} asks that the joint spectrum of the algebraic tuple $\bm T$ be an algebraic variety.  Let us note that for an algebraic tuple $\bm T$, the joint spectrum is not necessarily an algebraic variety. In fact, this is a rarity for algebraic tuples on an infinite dimensional Hilbert space.
\begin{example}
For a quick example, let $U$ be a unitary operator with infinitely many points in $\sigma(U)$. While the pair $(U,U)$ is algebraic, its joint spectrum is a proper closed subset of $\{(z,w)\in\mathbb C^2: z=w\}$ and thus fails to be an algebraic variety, as any proper Zariski-closed subset of an affine line must be finite.
\end{example}

It should also be noted that conversely, there is a commuting contractive pair whose Taylor joint spectrum is an algebraic variety but the pair is not even annihilated by any (admissible) analytic function.
\begin{example}\label{E:Volt}
Let $L^2$ be the standard Hilbert space of square integrable functions on the unit circle with respect to the arc-length measure. Consider the Hilbert space $H^2\otimes L^2$ and acting on it the commuting pair given by
$$
(T_1,T_2)=(M_z\otimes I_{L^2},I_{H^2}\otimes V)
$$where $V$ is the Volterra operator. It is known that $\sigma(V)=\{0\}$ and $\sigma_T(T_1,T_2)=\sigma(M_z)\times \sigma(V)=\overline{\bD}\times \{0\}$. Thus, its joint spectrum is an algebraic variety with the corresponding ideal being $<z_2>$, the ideal generated by $z_2$. We shall argue that $(T_1,T_2)$ cannot be annihilated by any function that is analytic in a neighborhood of the closed bidisk $\overline{\bD}^2$. Indeed, let $f$ be such a function. Write $f(z_1,z_2)=g(z_2) + z_1 h(z_1,z_2)$. It can be argued that the power series of both $g$ and $h$ around the origin is absolutely convergent in a neighborhood of $\overline{\bD}$ and $\overline{\bD}^2$, respectively. Let $u,v$ be arbitrary elements in $L^2$. Note that
\begin{align*}
\langle  f(T_1,T_2)(1\otimes u), (1\otimes v)\rangle_{H^2\otimes L^2}
=\langle  g(V) u, v \rangle_{L^2}.
\end{align*}Since $u,v$ are arbitrary in $L^2$, $f(T_1,T_2)=0$ implies $g(V)=0$. So $g$ must vanish in $\sigma(V)=\{0\}$. Writing $g(z_2)=z_2^Ng_1(z_2)$ for some analytic function $g_1$ not vanishing at $0$, we are left with $V^N=0$. It is well known that the Volterra operator is not nilpotent.

\end{example}

Our quest of understanding algebraic operator tuples or pairs in particular began in \cite{DasSauPAMS} and was inspired by the work of Agler, Knese, and McCarthy \cite{AKM} which initiates the study of algebraic isometric pairs.
\subsection{The main results}
We briefly discuss in turn seven main results of the paper; the first four results below are concerning the Cayley--Hamilton tuples.
\begin{enumerate}
\item[(\textbf{A})] It appears that a solution to Problem \ref{P:TheProb} is unknown even in the case of a commuting matrix tuple. While the papers \cite{Phil} and \cite{AS} focus on finding the polynomials that annihilate a given commuting matrix tuple, there appears to be no literature that extends the analysis to finding a solution to the more complete Problem \ref{P:TheProb}. We show that every commuting matrix tuple is a Cayley--Hamilton tuple; See \cref{Thm:CayleyHamilton} below. Indeed, for a matrix tuple $\bm T=(T_1,T_2,\dots, T_d)$, a collection $\cI$ of polynomials that solve Problem \ref{P:TheProb} for $\bm T$ is given by $\cI=\{p_\alpha:\alpha\in\bC^d\}$ where
$$
p_{\bm\alpha}(\bm z)=\det\Big( \alpha_1(T_1-z_1I)+\alpha_2(T_2-z_2I)+\cdots+\alpha_d(T_d-z_dI)\Big).
$$Moreover the collection $\cI$ can be chosen to be finite.
\item[(\textbf{B})] Let $\varphi_1, \dots, \varphi_{k}$ be contractive $N\times N$ matrix-valued rational functions with poles away from the closed polydisk $\overline{\bD}^d$ such that 
$$
\left[\varphi_i(\bm z), \varphi_j(\bm z)\right]:=\varphi_i(\bm z) \varphi_j(\bm z)-\varphi_j(\bm z)\varphi_i(\bm z)=0 \quad \text{ for all, } i, j \text { and } \bm z\in\overline{\bD}^d.
    $$ Then the tuple of Toeplitz operators $(M_{ z_1}, \dots, M_{z_d}, M_{\varphi_1}, \dots, M_{\varphi_k})$ acting on $H^2_{\bD^d}(\bC^N)$ is a Cayley--Hamilton tuple. See \cref{Poly1} below for further details. In this case also we argue that the collection $\cI$ can be taken to be finite.
\item[(\textbf{C})] 
Every commuting isometric tuple $\bm V=(V_1,V_2,\dots,V_d)$ such that with $V=V_1V_2\cdots V_d$, $V^{*n}\to0$ as $n\to\infty$ and $\dim\operatorname{Ran}(I-VV^*)<\infty$ is a Cayley--Hamilton tuple.
Moreover, the collection $\cI$ of polynomials that  solves Problem \ref{P:TheProb} for $\bm V$ is explicitly found in \cref{T:TheLastPiece} below. A functional model for commuting isometries due to Berger, Coburn and Lebow \cite{BCL}, briefly discussed in \S \ref{Pr1}, has proved indispensable. Indeed, it would not have been possible to find the collection $\cI$ explicitly if it were not for this result.
\item[(\textbf{D})] Let $\sigma(\cdot)$ be any notion of joint spectrum that enjoys the polynomial mapping theorem and that always contains the Taylor spectrum, i.e., for any polynomial $p$ and any commuting operator tuple $\bm A$, we have
$$
p(\sigma(\bm A))=\sigma(p(\bm A))\quad\mbox{and}\quad
\sigma_T(\bm A)\subset \sigma(\bm A).
$$
Then for any Cayley--Hamilton tuple $\bm T$, we have 
$$
\sigma_T(\bm T)=\sigma(\bm T)=\operatorname{supp}(\operatorname{Ann}(\bm T)).
$$In particular, for any Cayley--Hamilton tuple $\bm T$, we have
$$
\sigma_T(\bm T)=\sigma_W(\bm A)=\operatorname{supp}(\operatorname{Ann}(\bm T)),
$$where $\sigma_W(\cdot)$ denotes the Waelbroeck spectrum.
All this is proved in \S \ref{S:Annihilation}. Here the notion of the support of an ideal is as originally defined by Clou\^atre and Timko \cite{CT}.
\end{enumerate}
In order for an application of the results stated above, we now turn to the following special algebraic varieties that have invited researchers with interest in function theory, operator theory and complex geometry for good reasons.
\begin{definition}
Let $\cZ$ denote the common zero set of a number of non-zero polynomials in $d$--variables ($d\geq 2)$. We say that the algebraic variety $\cZ\cap\overline{\bD}^d$ is a \textit{distinguished variety} if 
$$
\cZ\cap\bD^d\neq\emptyset \quad\mbox{and}\quad \cZ\cap\partial\overline{\bD}^d=\cZ\cap\bT^d.
$$Here $\partial\overline{\bD}^d$ denotes the Euclidean boundary of $\overline{\bD}^d$ and $\bT^d=\bT\times\cdots\times\bT$ ($d$ times).
\end{definition}
Distinguished varieties dates back at least to Rudin’s 1969 work \cite{Rudin}, 
and was subsequently developed through the landmark contributions of Agler and McCarthy \cite{AM_Acta}. Since then, distinguished varieties 
have attracted considerable attention owing to their deep connections with several areas of mathematics. We mention a few of these connections here. From the perspective of function theory, distinguished varieties arise naturally as the
uniqueness sets for certain solvable Nevanlinna--Pick interpolation problems; see \cite{AM_Acta} (see also \cite{DS}) and the more recent work \cite{DKS}. From the viewpoint of operator theory, they play a crucial role in strengthening Ando’s inequality \cite{AM_Acta, DSJFA2017}. From the complex-geometric standpoint, building on Fedorov’s work on unramified point-separating pairs of inner functions, Vegulla \cite[Theorem~3.4.4]{Vegulla} proved that every planar domain bounded by piecewise analytic curves is conformally equivalent to a distinguished variety, thereby providing a unified framework for the function theory of multi-connected domains.
More recently, it has been shown that distinguished varieties are always one-dimensional 
manifolds with singularities, irrespective of the dimension of the ambient space, see  \cite{Pal}. From the operator-algebraic point of view, Guo, Wang and Zhao \cite{GWZ} established a connection between distinguished varieties and essentially normal quotient modules, particularly in relation to the Arveson--Douglas conjecture. Furthermore, 
distinguished varieties have appeared in generalisations of Wermer’s maximality theorem on the bidisk \cite[Theorem~1.2]{E}. Several algebraic, geometric, and operator-theoretic characterisations of distinguished varieties are now known; see, for  example, the original work \cite{AM_Acta} and the later developments \cite{BKS-APDE, GWZ, Knese-TAMS2010, Pal}. The following result elucidates the connection between distinguished varieties and the Taylor joint spectrum.

\begin{enumerate}
\item[(\textbf{E})] A subset $\cS$ of $\overline{\bD}^d$ is a distinguished variety if and only if 
$$
\cS=\sigma_T(V_1,V_2,\dots,V_d)
$$where $(V_1,V_2,\dots,V_d)$ is a commuting tuple of isometries such that $V_j^{*n}\to 0$ strongly as $n\to\infty$ and $\dim\operatorname{Ran}(I-V_jV_j^*)<\infty$ for each $j=1,2,\dots,d$. This is \cref{T:JointSpecDisVar} below.
\end{enumerate}
Let us mention that the heavy-lifting of the result above is borne by the result stated below and the work done in \cite{BKS-APDE} and in later developments \cite{DasSauPAMS} and \cite{Pal}. 
    \begin{enumerate}
\item[(\textbf{F})] 
Let $\varphi_1,\varphi_2,\dots,\varphi_k$ be matrix-valued rational functions on $\overline{\bD}^d$ as in item (\textbf{B}) above. Then 
$$
\sigma_T(M_{\varphi_1},M_{\varphi_2},\dots,M_{\varphi_k}) = \bigcup_{\bm z\in\overline{\mathbb{D}^d}}\sigma\big( \varphi_1(\bm z), \varphi_2(\bm z),\dots, \varphi_{k}(\bm z)\big).
$$This is \cref{Thm:M} below, obtained as a by-product of an in-depth analysis done for the Cayley--Hamilton tuple as described in item (\textbf{B}) above.
\end{enumerate}

For a bi-variate polynomial $\xi$, when does there exist a commuting isometric pair $(V_1,V_2)$ with finite dimensional defects so that $\cZ(\xi)\cap\overline{\bD}^2=\sigma_T(V_1,V_2)$? The result below provides an algebraic representation of $\xi$ which serves both as a necessary and sufficient condition for an affirmative answer. This can be seen as a partial converse of the main theme of the paper where, loosely speaking, we always start with a operator tuple $\bm T$ and ask when $\sigma_T(\bm T)$ is an algebraic variety. The result stated in item (\textbf{E}) is used in the result below.
\begin{enumerate}
\item[(\textbf{G})] Let $\xi$ be a polynomial in two variables. There exists a pure commuting isometric pair $(V_1,V_2)$ with finite
defect such that
$
\cZ(\xi)\cap\overline{\bD}^2=\sigma_T(V_1,V_2)
$ holds
if and only if there exist unimodular complex numbers $\alpha_i,\beta_j$ for $i=1,2,\dots,m$ and $j=1,2,\dots,n$ such that
\begin{align*}
\xi(z_1, z_2)
=
\prod_{i=1}^m (z_1-\alpha_i)\prod_{j=1}^n(z_2-\beta_j)\cdot \eta(z_1, z_2)\cdot \chi(z_1, z_2),
\end{align*}
where $\chi$ is a polynomial with no zeros on $\overline{\bD^2}$ and $\cZ(\eta)$ defines a distinguished variety. This is with the possibility that some of the factors above may be absent. This is \cref{T:AnyPoly} below.
\end{enumerate}

\section{Continuity of joint eigenvalues and a Cayley--Hamilton Theorem}\label{S:Prelim}
We begin by putting together some key tools. The final goal of this section is to prove the result stated in item (\textbf{A}). 

\subsection{A functional model for commuting isometries}\label{Pr1}
Here we briefly present a functional model for commuting isometries due to Berger, Coburn and Lebow \cite{BCL} (see also the paper \cite{BS-Helsinki} for some new insights into this model). The model is in terms of certain operator-valued linear pencils. For a Hilbert space $\cE$ and a bounded analytic function $\Psi:\bD^d\to \cB(\cE)$, the radial limit
$$
\lim_{r \to1-} \Psi(r\bm\zeta),\quad \bm\zeta\in\bT^d,
$$is known to exist almost everywhere on the $d$-tori $\bT^d$ in the strong operator topology (SOT).
Of particular importance to us are the matrix-valued rational \textit{inner} functions.  A $\cB(\cE)$-valued bounded analytic function is said to be \textit{inner}, if the radial limits are isometric operators. It is well known that in the cases when $\cE=\bC$ and $d=1$, the rational inner functions are precisely the finite Blaschke products. 

A bounded analytic function $\Psi:\bD^d\to\cB(\cE)$ gives rise to a \textit{multiplication operator} $M_{\Psi}: H^2_{\bD^d}(\cE) \rightarrow H^2_{\bD^d}(\cE)$ defined by $f \mapsto \Psi \cdot f$. It is well known that $\Psi$ is inner if and only if $M_{\Psi}$ is an isometry.
An operator $T$ acting on a Hilbert space is said to be \emph{pure} if
$$
{T^*}^n \to 0 \quad \text{in the strong operator topology (SOT).}
$$
When the multiplication operator $M_{\Psi}$ is pure, we say that the function $\Psi$ is pure. A very useful characterisation of pure contractive analytic functions on $\bD$ due to Bercovici--Douglas--Foias \cite{BDF} is the following:
\begin{align}\label{L:PureToeplitz}
\textit{$M_\Psi$ on $H^2_{\bD}(\cE)$ is a pure contraction if and only if $\Psi(0)$ is a pure contraction.}
\end{align}
We refer the readers to the recent work of Sarkar \cite{SrijanCAOT} for results analogous to \eqref{L:PureToeplitz} for general reproducing kernel Hilbert spaces.

For a contractive operator $T$, we shall denote 
$$
D_T=(I-T^*T)^{1/2}\quad\mbox{and}\quad\cD_T=\overline{\operatorname{Ran}}D_T.
$$We now state the functional model for commuting tuples of isometries due to Berger, Coburn, and Lebow \cite{BCL}. Consider a tuple of commuting isometries $(V_1, V_2, \dots, V_d)$ on a Hilbert space $\mathcal{H}$, and let $V := V_1 \cdots V_d$. By the classical Wold decomposition \cite{Wold}, the space $\mathcal{H}$ decomposes as
$$
\mathcal{H} = H^2(\mathcal{D}_{V^*}) \oplus \bigcap_{n\geq1}V^n\cH,
$$
and the operator $V$ decomposes as
$$
V = 
\begin{bmatrix}
	M_z & 0 \\
	0 & W
\end{bmatrix}
:\begin{bmatrix}
H^2(\mathcal{D}_{V^*}) \\ \bigcap_{n\geq1}V^n\cH
\end{bmatrix}
\to 
\begin{bmatrix}
H^2(\mathcal{D}_{V^*}) \\ \bigcap_{n\geq1}V^n\cH
\end{bmatrix}
$$
where $W$, the restriction of $V$ to $\bigcap_{n\geq1}V^n\cH$, is unitary. Berger, Coburn, and Lebow \cite[Theorem 3.1]{BCL} proved that for each $j = 1, \dots, d$, there exist a projection $P_j \in \mathcal{B}(\mathcal{D}_{V^*})$, a unitary operator $U_j \in \mathcal{B}(\mathcal{D}_{V^*})$, and commuting unitary operators $W_1, \dots, W_d$ on $\bigcap_{n\geq1}V^n\cH$ such that, under the same unitary identification, each $V_j$ has the form
\begin{align}\label{BCLform}
V_j = 
\begin{bmatrix}
	M_{P_j^\perp U_j + z P_jU_j} & 0 \\
	0 & W_j
\end{bmatrix}
:
\begin{bmatrix}
H^2(\mathcal{D}_{V^*}) \\ \bigcap_{n\geq1}V^n\cH
\end{bmatrix}\to\begin{bmatrix}
H^2(\mathcal{D}_{V^*}) \\ \bigcap_{n\geq1}V^n\cH
\end{bmatrix}.
\end{align}
\begin{thm}\label{T:BCL}
Suppose $(V_1, V_2, \dots, V_d)$ is a tuple of commuting isometries acting on a Hilbert space $\mathcal{H}$. The Wold decomposition of the product isometry $V := V_1 \cdots V_d$ decomposes each $V_j$ as \eqref{BCLform} for some unique orthogonal projections $P_j$, unitary operators $U_j$ in $\cB(\cD_{V^*})$, and commuting tuple of unitary operators $(W_1,W_2,\dots,W_d)$ acting on $\bigcap_{n\geq1}V^n\cH$. Moreover, when $d=2$, the model \eqref{BCLform} takes the simpler form
\begin{align}\label{BCLform2}
(V_1,V_2)=\left(
\begin{bmatrix}
	M_{P^\perp U + z PU} & 0 \\
	0 & W_1
\end{bmatrix},
\begin{bmatrix}
	M_{U^*P+z U^*P^\perp} & 0 \\
	0 & W_2
\end{bmatrix}
\right)
\end{align}
for some orthogonal projection $P$, unitary operator $U$ in $\cB(\cD_{V^*})$, and commuting pair of unitary operators $(W_1,W_2)$ acting on $\bigcap_{n\geq1}V^n\cH$. 
\end{thm}
The following class of commuting isometric tuples will be relevant later in the paper.
\begin{definition}\label{D:IsoDefect}
We shall say that a commuting isometric tuple $(V_1,V_2,\dots,V_d)$ 
\begin{itemize}
\item has \textit{finite defect} if $\dim\cD_{V^*}<\infty$ where $V=V_1V_2\cdots V_d$;
\item is \textit{pure}, if $V=V_1V_2\cdots V_d$ is pure, i.e., $V^{*n}\to 0$ in SOT as $m\to\infty$.
\end{itemize}
\end{definition}
We shall be dealing with operator tuples of the form $(M_{\varphi_1},M_{\varphi_2},\cdots,M_{\varphi_k})$. To ensure commutativity of this tuple, we shall need a condition on the symbols $\varphi_j$s. The following terminologies will be useful.
\begin{definition}\label{D:Commutator}
 We say that operator-valued analytic functions $\varphi_1, \dots, \varphi_{k}$ on $\bD^d$ for $d\geq1$, are {\em commuting} if for all $i,j=1,2,...,k$ and $\bm{z}\in\mathbb{D}^d$
    $$
    \left[\varphi_i(\bm z), \varphi_j(\bm z)\right]:=\varphi_i(\bm z) \varphi_j(\bm z)-\varphi_j(\bm z)\varphi_i(\bm z)=0.
    $$
    Of particular interest will be the following operator-valued linear polynomials
\begin{align}\label{varphis1}
\varphi_j(z)=P_j^\perp U_j+zP_jU_j 
\end{align}for some orthogonal projections $P_1, P_2,\dots, P_k$ and unitary operators $U_1, \cdots, U_k$. We shall call these the {\em BCL functions} if they are commuting and if, in addition, they satisfy
\begin{align}\label{varphis2}
\prod_{j=1}^k\varphi_j(z)=zI \quad\mbox{for every }z\in\bC.
\end{align}
\end{definition}
It should be noted that by the Berger--Coburn--Lebow (BCL) \cref{T:BCL}, pure commuting isometric tuples are precisely of the form $(M_{\varphi_1},M_{\varphi_2},\cdots,M_{\varphi_k})$ for some BCL functions. Also note that when a pure commuting isometric tuple has finite defect, then the corresponding BCL functions take values in the matrix space of size $\dim\cD_{V^*}$.

\subsection{ Joint Spectrum}\label{SS:JS} In this section we briefly discuss the Taylor spectrum.
The definition is somewhat involved, yet we find it imperative for the uninitiated readers to have a brief description about it since it will be used throughout the manuscript. Let $\mathcal{E}$ denote the exterior algebra generated by $d$ indeterminate objects $e_1, \dots, e_d$ with identity $e_0 \equiv 1$. This simply means that $\mathcal{E}$ consists of arbitrary linear combinations of products of the $e_i$s and the product is subject to the anti-commutation relation
\begin{align}\label{anti-commute}
e_i e_j =-e_j e_i, \quad 1 \le i,j \le d.
\end{align}
For each $i$, define the \emph{creation operator} $C_i : \mathcal{E} \to \mathcal{E}$ by $
C_i \, \xi = e_i \, \xi, \quad \xi \in \mathcal{E}. $
Upon declaring that the set
$$
\{ e_{i_1} \cdots e_{i_k} : 1 \le i_1 < \dots < i_k \le d \} 
$$ is an orthonormal basis, we can endow $\mathcal{E}$ with a Hilbert space structure. We then have the orthogonal decomposition
$$
\mathcal{E} = \bigoplus_{k=0}^d \mathcal{E}_k,
$$where for each $k\geq1$, $\cE_k=\operatorname{span}\{e_{i_1} \cdots e_{i_k} : 1 \le i_1 < \dots < i_k \le d \}$. It follows that $\dim \mathcal{E}_k = \binom{d}{k}$. Consequently, any $\xi \in \mathcal{E}$ admits a unique decomposition of the form
$\xi = e_i \, \xi' + \xi'',$ 
where neither $\xi'$ nor $\xi''$ contains a component along $e_i$. It then follows that $
C_i^* \, \xi = \xi', $
so that each $C_i$ is a partial isometry satisfying $
C_i^* C_j + C_j C_i^* = \delta_{ij}. $

Let $\cX$ be a normed space and $\bold{T} = (T_1, \dots, T_d)$ be a tuple of commuting bounded operators on $\cX$. Let
$\mathcal{E}(\cX) := \cX \otimes \mathcal{E},$ and define the operator $\Lambda_{\bold{T}} : \mathcal{E}(\cX) \to \mathcal{E}(\cX)$ by
$$
E_{\bold{T}} := \sum_{i=1}^d T_i \otimes C_i.
$$
It is straightforward to verify using \eqref{anti-commute} that $E_{\bold{T}}^2 = 0$. So we have $\operatorname{Ran} E_{\bm{T}} \subset \operatorname{Ker}E_{\bm{T}}$. The tuple $\bold{T}$ is said to be \emph{non-singular} on $\cX$ if
$\operatorname{Ran} E_{\bold{T}} = \operatorname{Ker} E_{\bold{T}}.$

\begin{definition}[Taylor joint spectrum]
	The \emph{Taylor joint spectrum} of $\bold{T}=(T_1, \cdots, T_d)$ on $\cX$ is defined as
$$
	\sigma_T(T_1, \dots, T_d) := \Big\{ \bold{\lambda} = (\lambda_1, \dots, \lambda_d) \in \mathbb{C}^d : \bold{T} - \bold{\lambda} \text{ is singular} \Big\},
$$
	where $\bold{T} - \bold{\lambda} := (T_1 - \lambda_1 I, \dots, T_n - \lambda_n I)$, and singularity means that the operator
$$
	E_{{\bold{T}} - \lambda} := \sum_{i=1}^d (T_i - \lambda_i I) \otimes C_i
$$
	fails to be non-singular, i.e., $\operatorname{Ran} E_{\bold{T} - \lambda} \neq \operatorname{Ker} E_{\bold{T} - \lambda}.$
\end{definition}
The decomposition $\mathcal{E} = \bigoplus_{k=0}^d \mathcal{E}_k$ gives rise to the \emph{Koszul complex} associated to $\bold{T}$ on $\cX$, defined by
\begin{align}\label{KoszulComplex}
K({\bold{T}}, \cX) : 0 \longrightarrow \mathcal{E}_0(\cX) 
\stackrel{D_0}{\longrightarrow} \mathcal{E}_1(\cX)
\stackrel{D_1}{\longrightarrow} \mathcal{E}_1(\cX)\cdots \mathcal{E}_{d-1}(\cX)
\stackrel{D_{d-1}}{\longrightarrow} \mathcal{E}_d(\cX) \longrightarrow 0,
\end{align}
where $D_k$ denotes the restriction of $E_{\bold{T}}$ to the subspace $\mathcal{E}_k(\cX)$. With this notation, the Taylor joint spectrum can equivalently be described as
$$
\sigma_T(\bm{T}) = \{ \lambda \in \mathbb{C}^d : K(\bold{T} - \lambda, \cX) \text{ is not exact} \}.
$$
The equivalence of these two formulations of the Taylor spectrum is explained in \cite{Cur}.

Let us denote by $\sigma_p(\bm T)$ the (possibly empty) set of point spectrum for $(\bm T)$, i.e.,
$$
\sigma_p(\bm T)
=\{\bm\lambda\in\bC^d: \mbox{there exists } 0\neq v \mbox{ such that } T_jv=\lambda_jv \mbox{ for each }j=1,2,\dots,d\}.
$$Then we always have the inclusions
\begin{align}\label{PointTaylor}
\sigma_p(\bm T) \subset \sigma_T(\bm T)\subset \sigma_W(\bm T),
\end{align}where $\sigma_W(\cdot)$ is the Waelbroeck spectrum defined by (see \cite{ArvesonII, W}) 
\begin{align}\label{definition_Wspectra}
\sigma_W(\bold{T}) := \{ \bm\lambda \in \mathbb{C}^d : p(\bm\lambda) \in \sigma(p(\bold{T})) \text{ for all polynomials } p \},
\end{align}
The first inclusion in \eqref{PointTaylor} is due to the fact that for a point spectrum member the exactness of the Koszul complex \eqref{KoszulComplex} breaks at stage $1$.  
The second inclusion can be seen to follow from the facts that both Taylor and Waelbroeck spectra enjoy the polynomial (more generally rational) functional calculus (see \cite{Cur} and \cite[Proposition 1.1.2]{W}), and that any notion of spectrum that enjoys the polynomial functional calculus must be contained in the Waelbroeck Spectrum. The latter fact can be checked easily.

If we denote $\bold{T}^{*} = (T_{1}^{*}, \ldots, T_{d}^{*})$, then the discussion in \cite[Section~3]{RS} 
shows that
\begin{align}\label{AdTaylor}
	\sigma_{\mathrm{T}}(\bold{T}^{*}) = \{(\overline{z_{1}}, \ldots, \overline{z_{d}}) : (z_{1}, \ldots, z_{d}) \in \sigma_{\mathrm{T}}(\bold{T})\}.
\end{align}
Lastly, it is worth mentioning that in the finite-dimensional case, the Waelbroeck spectrum $\sigma_W(T_1, \dots, T_d)$ is just the point spectrum $\sigma_p(T_1, \dots, T_d)$ (see, for example, \cite[Proposition 7]{MPR}), and thus by the inclusions \eqref{PointTaylor}, these three notions of joint spectrum coincide. For this reason, we shall make no distinction in the notation of the Taylor and Waelbroeck joint spectrum for matrices and simply denote it by $\sigma(\cdot)$.

\subsection{Continuity of joint eigenvalues and a Cayley--Hamilton Theorem}
Next, we come to a multi-variable analgoue of two key results in linear algebra, viz., the continuity of eigenvalues and the Cayley Hamilton Theorem. At the heart of the results in this subsection is the fact that every commuting (finite) family of matrices can be triangularized simultaneously by a unitary similarity transformation.

It is well-known that if $A_n$ is a sequence of matrices converging to a matrix $A$ and if $\lambda$ is an eigenvalue of $A$, then there is a strictly increasing sequence $\{n_k\}$ of natural numbers and $\lambda_{n_k}\in \sigma(A_{n_k})$ such that $\lambda_{n_k}\to \lambda$ as $k\to\infty$. The underlined facts used here are the facts that the eigenvalues are precisely the roots of a polynomial and that the roots of a polynomial depend continuously on its coefficients. The joint eigenvalues cannot, in general, be realized as the zero set of a single polynomial in $\mathbb{C}[z_1, z_2, \dots, z_d]$. Nevertheless, the joint eigenvalues enjoys a similar sort of continuity.

\begin{proposition}\label{P:JointContinuity}
Let $(A_{n1},A_{n2},\dots, A_{nd})$ be a sequence of commuting tuple of $N\times N$ matrices converging to $(A_1,A_2,\dots, A_d)$. If $(\lambda_1,\lambda_2,\dots,\lambda_d)$ is a joint eigenvalue for $(A_1,A_2,\dots, A_d)$, then there is an increasing sequence $\{n_k\}$ of natural numbers and $(\lambda_{n_k1},\lambda_{n_k2},\dots, \lambda_{n_kd})$ in $\sigma(A_{n_k1},A_{n_k2},\dots, A_{n_kd})$ such that for all $j=1,2,\dots,d$,
\begin{align*}
\lambda_{n_kj}\to \lambda_j \quad \mbox{as }k\to\infty.
\end{align*}
\end{proposition}
\begin{proof}
The crux of the matter is that commuting matrices can be simultaneously upper triangularized. Indeed, there is a sequence $\{U_n\}$ of unitary matrices such that 
$$
U_n^*(A_{n1},A_{n2},\dots, A_{nd})U_n = (T_{n1},T_{n2},\dots, T_{nd})
$$where $T_{nj}$ are upper triangular. Since $U_n$ are unitary, there is a converging subsequence $U_{n_k}$ converging to a unitary $U$, say. Then taking limit as $k\to\infty$ on both sides of 
$$
U_{n_k}^*(A_{n_k1},A_{n_k2},\dots, A_{n_kd})U_{n_k} = (T_{n_k1},T_{n_k2},\dots, T_{n_kd})
$$we get upper triangular matrices $T_j$ for $j=1,2,\dots, d$ such that
\begin{align*}
U^*(A_1,A_2,\dots, A_d)U=(T_1,T_2,\dots,T_d).
\end{align*}What remains is to recall that the joint eigenvalues are unitary invariant and that the joint eigenvalues appear on the diagonal entries of simultaneous triangularization.
\end{proof}
\begin{remark}
The authors could not find Proposition \ref{P:JointContinuity} in the existing literature. However, it can be seen to follow from a non-trivial result in \cite[page 92]{Elsner}. In contrast, the simple proof given above makes use of the simultaneous triangularisation of commuting matrices. This simple proof emerged out of a discussion between the last author and Mukul, an undergraduate student of IISER Pune.
\end{remark}
The following lemma will be used. \begin{lemma}\label{L:ZeroJoint}
Let $\bm A=(A_1,A_2,\dots,A_d)$ be a commuting tuple of $N\times N$ matrices. If
\begin{align}\label{ZeroJoint}
\det(\alpha_1A_1+\alpha_2A_2+\cdots+\alpha_dA_d)=0
\end{align}for all $\bm\alpha=(\alpha_1,\alpha_2,\dots,\alpha_d)$ coming from a  set $\cS$ of cardinality at least $N(d-1)+1$ such that each subset of cardinality $d$ is linearly independent, then $\sigma(\bm A)$ must contain $(0,0,\dots,0)$. In particular, if \eqref{ZeroJoint} holds for all $\bm\alpha$ in $\bC^d$, then the conclusion holds.
\end{lemma}
\begin{proof}
Since $(A_1,A_2,\dots,A_d)$ is a commuting tuple, we can suppose without loss of generality that each of the matrices are upper triangular. Then the joint eigenvalues appear in the diagonal entries of $(A_1,A_2,\dots,A_d)$. Suppose
$$
\sigma(A_1,A_2,\dots,A_d)=\{\bm\lambda_j=(\lambda_j^{(1)},\lambda_j^{(2)},\dots,\lambda_j^{(d)}): j=1,2,\dots,N\}.
$$The condition \eqref{ZeroJoint} means that
\begin{align}\label{PigeonHole}
\alpha_1\lambda_j^{(1)}+\alpha_2\lambda_j^{(2)}+\cdots+\alpha_d\lambda_j^{(d)}=0
\end{align}for every $\bm\alpha$ in $\cS$ and for some $j=1,2,\dots,N$ depending on $\bm\alpha$. The cardinality condition on $\cS$ ensures, by pigeonhole principle, that there is at least one joint eigenvalue, say $\bm\lambda_1$, so that \eqref{PigeonHole} is satisfied by $d$ points $\bm\alpha$ in $\cS$. Since any $d$ elements of $\cS$ is linearly independent, this will force $\bm\lambda_1$ to be the zero vector. 
\end{proof}
\begin{thm}\label{Thm:CayleyHamilton}
Every commuting matrix-tuple is a Cayley--Hamilton tuple. More precisely,
for a commuting tuple $\bm  A=(A_1,A_2,\dots,A_d)$ of $N\times N$ matrices, set
\begin{align}\label{p_alpha}
p_{\bm\alpha}(z_1,z_2,\dots,z_d)=\det\bigg( \alpha_1(A_1-z_1I)+\alpha_2(A_2-z_2I)+\cdots+\alpha_d(A_d-z_dI) \bigg).
\end{align}Then 
\begin{align}\label{CayleyHamilton}
p_{\bm\alpha}(\bm A)=0
\quad\mbox{for every $\bm\alpha$ in $\bC^d$},
\end{align}and
\begin{align}\label{SmallerWael}
\sigma(\bm A)= \{\bm\lambda\in\bC^d: p_{\bm\alpha}(\bm\lambda)=0 \mbox{ for all }\bm\alpha\in\bC^d\}.
\end{align}

Moreover, the conclusion remains true even when the vectors $\bm\alpha$ run over a significantly smaller set, viz., any set of cardinality at least $N(d-1)+1$ such that each subset of cardinality $d$ is linearly independent.
\end{thm}
\begin{proof}
To see that each $p_{\bm\alpha}$ annihilates $\bm A$, we use the simultaneous triangularization of commuting matrices as follows. Suppose without loss of generality that $\bm A$ is in say upper triangular form with respect to the orthogonal basis $\{u_1,u_2,\dots,u_d\}$, and the $i$-th diagonal entry for $A_j$ being $\lambda_j^i$, i.e., the list of diagonal vectors for $\bm A$ is given by
$$
\{\bm\lambda_j=(\lambda_j^{(1)},\lambda_j^{(2)},\dots,\lambda_j^{(d)}): j=1,2,\dots,N\}.
$$Although we do not use it in this proof, the above set constitutes the joint eigenvalues for $\bm A$. For $\bm\alpha$ in $\bC^d$ fixed, the polynomial $p_{\bm\alpha}$ is then given by
\begin{align*}
p_{\bm\alpha}(z_1,z_2,\dots,z_d)=
\prod_{j=1}^d\bigg(\alpha_1(\lambda_j^{(1)}-z_1)+\alpha_2(\lambda_j^{(2)}-z_2)+\cdots +\alpha_d(\lambda_j^{(d)}-z_d)\bigg).
\end{align*}For notational convenience, let us denote the $j$-th factor above by $E_j(z_1,z_2,\dots,z_d)$. Let us note that the upper triangular matrix $E_j(\bm A)$ has the $j$-th diagonal entry $0$ and so
$$
\operatorname{Ran}E_d(\bm A) \subset \bigvee\{u_j:j=1,2,\dots,d-1\}:=\cE_{d-1},
$$and
$$
\operatorname{Ran}E_{d-1}(\bm A)|_{\cE_{d-1}} \subset \bigvee\{u_j:j=1,2,\dots,d-2\}:=\cE_{d-2},
$$More generally, for $j=1,2,\dots,d$,
$$
\operatorname{Ran}E_{j}(\bm A)|_{\cE_{j}} \subset \bigvee\{u_j:j=1,2,\dots,d-2\}:=\cE_{j-1},
$$where $\cE_0:=\{0\}$ and
$$
\cE_j:=\bigvee\{u_j:j=1,2,\dots,j\}\quad\mbox{for }j=1,2,\dots,d-1.
$$Consequently,
\begin{align*}
\operatorname{Ran}p_{\bm\alpha}(\bm A)&=\operatorname{Ran} \bigg(E_1(\bm A)\cdots E_d(\bm A)\bigg)\\
&\subset \operatorname{Ran}E_1(\bm A)|_{\cE_1}=\{0\}
\end{align*}showing that $p_{\bm\alpha}(\bm A)=0$.
Since $p_{\bm\alpha}(\bm A)=0$ for each $\bm\alpha$, it follows that $\sigma(\bm A)$ is contained in 
$$
\{\bm\lambda\in\bC^d: p_{\bm\alpha}(\bm\lambda)=0 \mbox{ for all }\bm\alpha\in\bC^d\}.
$$Since for commuting matrices, the spectrum consists of joint eigenvalues (observed in the discussion following \eqref{AdTaylor}), the proof of \eqref{SmallerWael} will be complete if we can show that every point in the set above
is a joint eigenvalue. Note that if $\bm\lambda=(\lambda_1,\lambda_2,\dots,\lambda_d)$ is in the set above, then the matrix tuple $(A_1-\lambda_1I,A_2-\lambda_2I,\dots,A_d-\lambda_dI)$ satisfies the hypothesis of Lemma \ref{L:ZeroJoint} and thus the zero vector is a joint eigenvalue for $(A_1-\lambda_1I,A_2-\lambda_2I,\dots,A_d-\lambda_dI)$, or equivalently, $\bm\lambda$ is a joint eigenvalue for $\bm A$.
Finally, the pigeonhole principle argument as in the proof of Lemma \ref{L:ZeroJoint} gives the moreover part.
\end{proof} 

\begin{remark}
Let us note the following result due to Phillips \cite{Phil} in connection with the identities \eqref{CayleyHamilton}. Phillips showed that \textit{if $(X_0,X_1\dots,X_{d})$ and $(Y_0,Y_1,\dots, Y_{d})$ are two families of square matrices such that $X_iX_j=X_jX_i$ for each $i,j=0,1,\dots,d$ and
\begin{align}\label{Philips'Hypo}
Y_0X_0+Y_1X_1+ \cdots+Y_{d}X_{d}=0,
\end{align}
then $(X_0,X_1,\dots, X_{d})$ is annihilated by the polynomial 
\[
p(z_0,z_1,\dots,z_{d})=\textup{det}(z_0Y_0+z_1Y_1+\cdots+z_{d}Y_{d}).
\]}New proofs of Philips' result are now available in the literature; see the recent paper \cite{AS} and the references therein. The identities \eqref{CayleyHamilton} can be seen to follow from Philips' result. Indeed, to show that the commuting square matrix tuple $(A_1,A_2,\dots,A_d)$ satisfies the polynomial
$p_{\bm\alpha}$ as in \eqref{p_alpha}, 
we fix an $\bm\alpha=(\alpha_1,\alpha_2,\dots,\alpha_d)$ in $\bC^d$ and take 
$$
(X_0,X_1,\dots,X_d)=(I, A_1,\dots, A_d)
$$ and 
\[
(Y_0,Y_1,\dots,Y_d)=(\alpha_1A_1+\alpha_1A_1+\cdots+\alpha_dA_d, -\alpha_1 I,\dots, -\alpha_d I).
\]
Note that the tuples $(X_0,X_1,\dots,X_d)$ and $(Y_0,Y_1,\dots,Y_d)$ satisfy Philips' hypothesis \eqref{Philips'Hypo} and so $(I, A_1,\dots, A_d)$ is annihilated by the polynomial
\begin{align*}
    p(z_0,z_1,\dots,z_{d})=\textup{det}\bigg(z_0(\alpha_1A_1+\alpha_1A_1+\cdots+\alpha_dA_d))-\alpha_1 z_1I-\cdots-\alpha_dz_dI\bigg),
\end{align*}
which means the same thing as in \eqref{CayleyHamilton} because $p(1,z_1,\dots,z_d)$ is the same polynomial as in \eqref{CayleyHamilton}.
\end{remark}

For matrix-valued rational symbols $\Psi_j$, we shall compute the (Taylor and Waelbroeck) joint spectrum of the tuple $(M_{\Psi_1},M_{\Psi_2},\dots,M_{\Psi_d})$.
The following single variable result will be used. The authors were unable to find a proof of this, so a proof is included. 
\begin{lemma}\label{Lemma:Matrix-single}
Let $\Psi:\bD\to M_N(\bC)$ be a rational function with poles away from $\overline{\bD}$. Then 
$$
\sigma(M_\Psi) = \bigcup_{z\in\overline{\bD}} \sigma(\Psi(z)).
$$
\end{lemma}
\begin{proof}
Suppose $\lambda\in\sigma(\Psi(z_0))$ for some $z_0\in\bD$. Let $v$ be non-zero vector in $\bC^N$ such that $\Psi(z_0)^*v=\overline{\lambda}v.$
This shows that 
$$
M_\Psi^*(k_{z_0}\otimes v)=k_{z_0}\otimes \Psi(z_0)^*v = \overline{\lambda}(k_{z_0}\otimes v),
$$which means $k_{z_0}\otimes v$ is an eigenvector for $M_\Psi^*$ corresponding to $\overline{\lambda}$. Thus we have
$$
 \bigcup_{z\in\bD} \sigma(\Psi(z))\subseteq\sigma(M_\Psi).
$$Since $\sigma(M_\Psi)$ is a compact set, it will contain the the closure of the union above. Now suppose that $\lambda\in\Psi(\zeta)$ for some $\zeta\in\mathbb{T}$. Let $r_n$ be an increasing sequence converging to $1$. Then consider the sequence $\Psi(r_n\zeta)$ of matrices converging to $\Psi(\zeta)$. By Continuity of Eigenvalues, $\lambda$ must be the limit point of eigenvalues of $\Psi(r_n\zeta)$ possibly along a subsequence. Thus, $\lambda$ must be in the closure of the union above and hence in $\sigma(M_\Psi)$.

It remains to establish the containment $\sigma(M_\Psi)\subseteq\bigcup_{z\in\overline{\bD}} \sigma(\Psi(z))$. Note that $M_\Psi-\lambda I = M_{\Psi-\lambda}$ and so it is enough to show that if $0\in \sigma(M_\Psi)$, then $0\in\bigcup_{z\in\overline{\bD}} \sigma(\Psi(z))$. Suppose, on the contrary, that $0$ does not belong to $\bigcup_{z\in\overline{\bD}} \sigma(\Psi(z))$. The goal is to show that $z\mapsto \Psi(z)^{-1}$ is bounded on $\bD$. Once we achieve this, $M_{\Psi^{-1}}$ would be a bounded inverse of $M_\Psi$. The hypothesis that each $\Psi(z)$ is non-singular implies that the continuous function
$$
\overline{\bD}\ni z \mapsto \det\Psi(z)
$$is non-vanishing on the compact set and therefore there must exist $\alpha>0$ such that $|\det \Psi(z)|>\alpha$ for every $z\in\overline{\bD}$. Note the formula for the rational function $\Psi(z)^{-1}$ 
$$
\Psi(z)^{-1} = \frac{1}{\det\Psi(z)} \operatorname{Adj}\Psi(z).
$$For a bounded function $\Psi(\cdot)$, $\operatorname{Adj}\Psi(\cdot)$ is also bounded. One way to see this is to note that each entry of the adjugate matrix is bounded by a certain number independent of the entry. Thus $\Psi(z)^{-1}$ is the product of two bounded function, and hence it is bounded.
\end{proof}

\section{The proofs of the main results} \label{S:MainSec}
This section contains the proofs of all the main results stated in items (\textbf{B})-(\textbf{G}).
The section is divided into several subsections and the results are proved in turn.
\subsection{Joint spectrum for commuting Toeplitz operators}
Here we prove the results stated in items (\textbf{B}) and (\textbf{F}). We begin with the following lemma.
\begin{lemma}\label{L:TheUnion}
Let $\varphi_1, \dots, \varphi_{k}$ be commuting (refer to Definition \ref{D:Commutator}), contractive $N\times N$ matrix-valued rational functions with poles away from the closed polydisk $\overline{\bD}^d$. Then
    \begin{align}\label{TheUnion1}
    \bigcup_{\bm z\in\overline{\mathbb{D}^d}}\sigma\big( \varphi_1(\bm z),\varphi_2(\bm z),\dots, \varphi_{k}(\bm z)\big)\subset  \sigma_T(M_{\varphi_1}, M_{\varphi_2}, \dots, M_{\varphi_k}) \subset
    \sigma_W(M_{\varphi_1},\dots, M_{\varphi_k}).
    \end{align}Moreover, the following are equivalent, where $\xi$ is a polynomial in $d$ variables:
    \begin{enumerate}
\item  $\sigma_W(M_{\varphi_1},M_{\varphi_2}, \dots, M_{\varphi_k})\subset \cZ(\xi)$;

\item $\sigma_T(M_{\varphi_1},M_{\varphi_2}, \dots, M_{\varphi_k})\subset \cZ(\xi)$;

\item $\bigcup_{\bm z\in\overline{\mathbb{D}^d}}\sigma\big( \varphi_1(\bm z),\varphi_2(\bm z),\dots, \varphi_{k}(\bm z)\big)\subset \cZ(\xi)$; and 

\item $\xi^r(M_{\varphi_1},M_{\varphi_2}, \dots, M_{\varphi_k})=0$ for some $r\in\mathbb{ N}$. 
    \end{enumerate}
    Moreover, the integer $r$ in item $(4)$ above can be chosen such that $r \leq N$.
\end{lemma}
\begin{proof}
For the containment \eqref{TheUnion1}, it is convenient to work with the adjoint operator. Let $\bm\lambda$ be in $\bD^d$ and
	$(\overline{\mu}_1,\overline{\mu}_2, \ldots, \overline{\mu}_k)$ be a joint eigenvalue of $\big(\varphi_1(\bm\lambda)^*,\varphi_2(\bm\lambda)^*, \ldots, \varphi_k(\bm\lambda)^*\big)$ with $u$ as a joint eigenvector. 
	Let $K$ be the Szeg\"o kernel on $\bD^d$. Then it is well known that 
	$$
	M_{\varphi_j}^*(K_{\bm\lambda} \otimes u) 
	= K_{\bm\lambda} \otimes \varphi_j(\bm\lambda)^* u
	= \overline{\mu}_j\,(K_{\bm\lambda} \otimes u).
	$$
	Therefore,  $(\overline{\mu}_1,\overline{\mu}_2,\dots, \overline{\mu}_k)$ is in the point spectrum of  $\left( M_{\varphi_1}^*, M_{\varphi_2}^*,\dots, M_{\varphi_{k}}^*\right)$. Thus, by virtue of 
    \eqref{AdTaylor}, we get
	$$
	\bigcup_{\bm{z} \in {\mathbb{D}^d}} \sigma\left({\varphi}_1(\bm{z}), {\varphi}_2(\bm{z}),\dots, \varphi(\bm z_d)\right)\subset \sigma_{T}( M_{\varphi_1}, M_{\varphi_2},\dots, M_{\varphi_k}). 
	$$
	The continuity of the joint spectrum (i.e., Proposition \ref{P:JointContinuity}), and the continuity of the symbols $\varphi_j$ on $\overline{\bD}^d$ yield the first containment in \eqref{TheUnion1}. The second containment is a generality already observed in \eqref{PointTaylor}.

Let us next note that the implications $(1)\Longrightarrow (2)\Longrightarrow (3)$ are easy consequences of \eqref{TheUnion1}. The implication $(4)\Longrightarrow(1)$ is easy because if $(\lambda_1, \lambda_2, \dots, \lambda_k)$ is an element of $\sigma_W(M_{\varphi_1}, M_{\varphi_2}, \dots, M_{\varphi_k})$, then, by the definition of the Waelbroeck spectrum, we get 
$$\xi^r(\lambda_1,\lambda_2, \dots, \lambda_k)\in \sigma(\xi^r(M_{\varphi_1}, M_{\varphi_2},\dots, M_{\varphi_k}))=\sigma (0)=\{0\}.$$
Therefore $\sigma_W(M_{\varphi_1},M_{\varphi_2}, \dots, M_{\varphi_k})\subset \cZ(\xi)$. So for the equivalence of the statements $(1)-(4)$, what remains is to show $(3)\Longrightarrow(4)$.
\smallskip

\noindent
{\sf Proof of $(3)\Longrightarrow (4)$:} By $(3)$, we have $$\sigma(\xi(\varphi_1(\bm z), \varphi_2(\bm z),\dots, \varphi_k(\bm z)))=\xi(\sigma(\varphi_1(\bm z), \varphi_2(\bm z),\dots, \varphi_k(\bm z)))=\{0\}$$ for all $\bm z\in\overline{\bD^d}$.  Since $(\varphi_1(\bm z), \varphi_2(\bm z),\dots, \varphi_k(\bm z))$ is a commuting $N\times N$ matrix tuple, the above observation shows that $\xi(\varphi_1(\bm z),\varphi_1(\bm z), \dots, \varphi_k(\bm z))$ is a nilpotent matrix and we are assured of an integer $r\leq N$ so that 
$$\xi^r(\varphi_1(\bm z),\varphi_2(\bm z), \dots, \varphi_k(\bm z))=\left[\xi(\varphi_1(\bm z),\varphi_2(\bm z), \dots, \varphi_k(\bm z))\right]^r=0.$$
Note that the integer $r=r(\bm z)$ above will, a priori, depend on $\bm z\in\overline{\bD}^d$ but to avoid this technicality, we can take $r=\max\{r(\bm z):\bm z\in\overline{\bD}^d\}$, which is not greater than $N$.
\end{proof}

\begin{thm}\label{Poly1}
	Let $\varphi_1, \varphi_2,\dots, \varphi_{k}$ be commuting and contractive $N\times N$ matrix-valued rational functions with poles away from the closed polydisk $\overline{\bD}^d$ for $d\geq1$. Then 
	\begin{align}\label{desired identity}
	\notag	\sigma_{T}(M_{ z_1}, \dots, M_{z_d}, M_{\varphi_1}, \dots, M_{\varphi_k})
		&= \sigma_W(M_{ z_1},\dots, M_{z_d}, M_{\varphi_1}, \dots, M_{\varphi_k})\\
		&=	\bigcup_{\bm z\in\overline{\mathbb{D}^d}}\sigma\big({z}_1I, \dots, {z}_dI, \varphi_1(\bm z),\dots, \varphi_{k}(\bm z)\big)
	\end{align}
 is an algebraic variety of dimension $d$ without isolated points. In particular, the tuple $(M_{ z_1}, \dots, M_{z_d}, M_{\varphi_1}, \dots, M_{\varphi_k})$ is a Cayley--Hamilton tuple.
\end{thm}
\begin{proof} We shall first establish the equalities in \eqref{desired identity}.
For $\bm z\in\overline{\bD^d}$, we shall use the short-hand notation
\[
(\bm{z}I, \boldsymbol{\varphi}(\bm{z})) := \big(z_1 I, \ldots, z_d I, \varphi_1(\bm{z}), \ldots, \varphi_k(\bm{z})\big)
\]and
\[
(\bm M_{\bm z},\bm M_{\bm\varphi}):=(M_{ z_1}, \dots, M_{z_d}, M_{\varphi_1}, \dots, M_{\varphi_k}).
\]
Let us also denote
	\begin{align}\label{TheUnion}
	\Omega := \bigcup_{\bm{z} \in \overline{\mathbb{D}^d}} \sigma\big(\bm{z}I, \boldsymbol{\varphi}(\bm{z})\big)\subset\overline{\bD^{d}}\times\overline{\bD^{k}}.
	\end{align}
By Lemma \ref{L:TheUnion}, we already have
    \begin{align*}
\Omega=\bigcup_{\bm{z} \in \overline{\mathbb{D}^d}} \sigma\big(\bm{z}I, \boldsymbol{\varphi}(\bm{z})\big) \subset \sigma_{T}(\bm M_{\bm z}, \bm M_{\bm \varphi})\subset \sigma_W(\bm M_{\bm z},\bm M_{\bm \varphi}).
   \end{align*}
    Therefore to complete the proof of the equalities in \eqref{desired identity}, we just need to prove
   \begin{align}\label{HowThough}
\sigma_W(\bm M_{\bm z},\bm M_{\bm \varphi})\subset \bigcup_{\bm{z} \in \overline{\mathbb{D}^d}} \sigma\big(\bm{z}I, \boldsymbol{\varphi}(\bm{z})\big)=\Omega.
   \end{align}So suppose $(\bm\lambda_1,\bm\lambda_2)$ is in $\sigma_W(\bm M_{\bm z},\bm M_{\bm\varphi})$. Note that for $(\bm\lambda_1,\bm\lambda_2)$ to be in the union as in \eqref{HowThough}, we must find a point $\bm z_0\in \overline{\bD}^d$ so that $(\bm\lambda_1,\bm\lambda_2)$ belongs to $\sigma\big(\bm{z}_0I, \boldsymbol{\varphi}(\bm{z}_0)\big)$. It is then immediate that $\bm z_0=\bm\lambda_1$. Note that here we make use of the cooridinate multipliers $\bm M_{\bm z}$. Thus we must show that $(\bm\lambda_1,\bm\lambda_2)\in\sigma\big(\bm\lambda_1I, \boldsymbol{\varphi}(\bm\lambda_1)\big)$ or equivalently, we have to show that
   \begin{align}
(\lambda_{21},\lambda_{22},\dots,\lambda_{2k})=:\bm\lambda_2\in\sigma(\bm\varphi(\bm\lambda_1))=\sigma(\varphi_1(\bm\lambda_1),\varphi_2(\bm\lambda_1),\dots,\varphi_k(\bm\lambda_1)).
   \end{align}Since $\varphi_j(\bm\lambda_1)$ are matrices, by Lemma \ref{L:ZeroJoint}, it is enough to show that
   \begin{align}\label{ZeroDet}
\det\bigg
(\alpha_1(\varphi_1(\bm\lambda_1)-\lambda_{21}I)+\alpha_2(\varphi_2(\bm\lambda_1)-\lambda_{22}I)+\cdots+\alpha_k(\varphi_k(\bm\lambda_1)-\lambda_{2k}I)\bigg)=0
   \end{align}
   for all $(\alpha_1, \alpha_2, \dots, \alpha_k)\in\bC^k$. Since $\varphi_j$ are assumed to be rational, let 
   $$
   \varphi_j(\bm z)=\frac{F_j(\bm z)}{q_j(\bm z)} \quad\mbox{for }j=1,2,\dots,k
   $$be its minimal representation in the sense that $q_j$ does not divide each entry of the matrix polynomial $F_j$. Consider the polynomials
   \begin{align}\label{SomePoly}
   p_i(\bm z) := \prod_{i\neq j=1}^k q_j(\bm z)\quad\mbox{and}\quad q(\bm z):=\prod_{j=1}^k q_j(\bm z).
   \end{align}
   With these notations in place, \eqref{ZeroDet} becomes
   \begin{align*}
   \det\bigg
(&\alpha_1p_1(\bm\lambda_1)(F_1(\bm\lambda_1)-\lambda_{21}q_1(\bm\lambda_1)I)\\
&+\alpha_2p_2(\bm\lambda_1)(F_2(\bm\lambda_1)-\lambda_{22}q_2(\bm\lambda_1)I)+\cdots+\alpha_kp_k(\bm\lambda_1)(F_k(\bm\lambda_1)-\lambda_{2k}q_k(\bm\lambda_1)I)\bigg)=0.
   \end{align*}
   For $\bm\alpha\in\bC^k$, consider the $(d+k)$-variable polynomial
   \begin{align}\label{ThePs}
   \notag 
   p_{\bm\alpha}(\bm z_1,\bm z_2)=
   \det\bigg
(\alpha_1p_1(\bm z_1)(F_1(\bm z_1)-z_{21}q_1(\bm z_1)I)
&+\alpha_2p_2(\bm z_1)(F_2(\bm z_1)-\lambda_{22}q_2(\bm z_1)I)\\
&+\cdots+\alpha_kp_k(\bm z_1)(F_k(\bm z_1)- z_{2k}q_k(\bm z_1)I)\bigg)
   \end{align}where $\bm z_1\in\bC^d$ and $\bm z_2=(z_{21},z_{22},\dots,z_{2k})\in\bC^k$. Since $(\bm\lambda_1,\bm\lambda_2)\in \sigma_W(\bm M_{\bm z},\bm M_{\bm\varphi})$, we must have
   $$
   p_{\bm\alpha}(\bm\lambda_1,\bm\lambda_2) \in \sigma(p_{\bm\alpha}(\bm M_{\bm z},\bm M_{\bm\varphi})) =\sigma( M_{p_{\bm\alpha}(\bm zI,\bm\varphi(\bm z))}).
   $$We now argue that
   \begin{align}\label{ToShow}
   p_{\bm\alpha}(\bm zI,\bm\varphi(\bm z))=0 \quad\mbox{for every }\bm z\in\bD^d.
   \end{align}This would yield $p_{\bm\alpha}(\bm\lambda_1,\bm\lambda_2)=0$, which is what was needed to show. For every $\bm z_1\in\bD^d$, let us look at the polynomial
   \begin{align*}
       r_{\bm\alpha,\bm z_1}(\bm z_2)&:=\det\bigg((\alpha_1(\varphi_1(\bm z_1)-z_{21}I)+\alpha_2(\varphi_2(\bm z_1)-z_{22}I)+\cdots+\alpha_k(\varphi_k(\bm z_1)-z_{2k}I)\bigg)\\
       &= \frac{p_{\bm\alpha}(\bm z_1,\bm z_2)}{q(\bm z_1)^N}
   \end{align*}where $q$ is the polynomial as in \eqref{SomePoly}. For a fixed $\bm z_1\in\bD^d$, apply the Cayley--Hamilton \cref{Thm:CayleyHamilton} to the commuting matrices $(\varphi_1(\bm z_1), \varphi_2(\bm z_1),\dots,\varphi_k(\bm z_1))$ to conclude that
   $$
   r_{\bm\alpha,\bm z_1}(\bm\varphi(\bm z_1)) = \frac{p_{\bm\alpha}(\bm z_1,\bm \varphi(\bm z_1)}{q(\bm z_1)^N}=0.
   $$Since $q(\bm z_1)$ never vanishes in $\overline{\bD}^d$, we must have
   $$
   p_{\bm\alpha}(\bm z_1,\bm\varphi(\bm z_1))=0
   $$for every fixed $\bm z_1\in\bD^d$. What remains is to observe that for every commuting matrix $k$-tuple $\bm A$ and $\bm z_1\in\bD^d$,
   $$
   p_{\bm\alpha}(\bm z_1I,\bm A)= p_{\bm\alpha}(\bm z_1,\bm A)\otimes I.
   $$ Take $\bm A=\bm\varphi(\bm z_1)$ to get \eqref{ToShow}.
   We next show that
   \smallskip

   \noindent
   {\sf $\Omega$ as in \eqref{TheUnion} is a variety:} 
	To prove this, consider the following subset of polynomials in $d+k$ variables:
    $$
    \cS:=\left\{f_{\bm\alpha}(\bm z, \bm w )=\det\big(\alpha_1\big(\varphi_1(\bm z)-w_1I\big)+\dots +\alpha_k\big(\varphi_{k}(\bm z)-w_{k}I\big)\big):\bm\alpha\in\bC^k\right\}
    .$$The analysis given below will work equally well for a significantly smaller subset of the set $\cS$ above where $\bm\alpha$ runs over a set of at least $N(k-1)+1$ many elements so that each of its subset with $k$ many elements is linearly independent. We work with the larger set $\cS$ for simplicity. We show that
    \begin{align}\label{Omega}
    \Omega=\cZ(\cS)\cap\overline{\bD}^{d+k}.
    \end{align}
    Suppose $(\lambda_1, \dots, \lambda_{d}, \mu_1, \dots, \mu_k)\in\Omega$. Again making use of the cooridinate multipliers $\bm M_{\bm z}$, this means that for each $j=1,\dots, {k}$, there exists a non-zero vector $v$ such that 
	$\varphi_j(\bm \lambda)v={\mu_j}v$, where $\bm\lambda=(\lambda_1,\lambda_2,\dots, \lambda_d)$. This implies that for each $\bm\alpha$ in $\bC^k$,
$$\sum_{j=1}^{k}\alpha_{j}\big(\varphi_j(\bm\lambda)-\mu_{j}I\big)v=0.$$
    This shows that $(\lambda_1,\dots, \lambda_{d}, \mu_1, \dots, \mu_k)\in \cZ(\cS)$, and hence $\Omega\subset \cZ(\cS)$. To prove the other containment, let $(\lambda_1, ,\dots, \lambda_{d}, \mu_1, \dots, \mu_k)\in \cZ(\cS)\cap\overline{\bD^d}\times\overline{\bD^k}$. This is the same as 
	$$
	\det\!\Big( \alpha_1 \big(\varphi_1(\bm \lambda) - w_1\big) 
	+ \cdots 
	+ \alpha_k \big(\varphi_{k}(\bm \lambda) - w_{k}\big) \Big) = 0,
	$$ 
    for all $\bm\alpha\in\bC^k$. Lemma \ref{ZeroJoint}, ensures that the tuple of matrices 
$$
\big(\varphi_1(\bm \lambda) - w_1 I,\;  \dots,\; \varphi_k(\bm \lambda) - w_{k} I\big)
$$
has $(0, \dots, 0)$ as a joint eigenvalue. Hence  $(\lambda_1,\dots, \lambda_{d}, \mu_1, \dots, \mu_k)\in\Omega.$ Therefore, $\Omega$ is a variety and so by a result of Eisenbud and Evans \cite{EG}, it coincides with the common zero set of $d+k$ polynomials in $d+k$ variables.

Thus we have shown that the tuple $(\bm M_{\bm z},\bm M_{\bm\varphi})$ is a Cayley--Hamilton tuple with $\cI=\{p_{\bm\alpha}:\bm\alpha\in\bC^k\}$ where $p_{\bm\alpha}$ are the polynomials as in \eqref{ThePs}.	 
	Next, we prove the additional information that 
    \smallskip
    
    \noindent
    {\sf $\Omega$ has dimension $d$ and has no isolated points:}  To that end, let $(\bm\lambda, \bm\mu)\in\overline{\bD^d}\times\overline{\bD^k}$ be a regular point of variety $\Omega$. For the uninitiated readers, \textit{regular points} are those where the Jacobian of the determining polynomial functions has the maximum rank; and the \textit{singular points} are the non-regular points. Consider the set $$\Delta_{\bm\lambda}=\{(\bm\lambda, \bm w): \bm w=(w_1, \dots, w_k)\in\bC^k \}.$$ By Theorem $8.1$ of \cite{RB}, we get
	\begin{align}\label{dimEq}
		\dim\big(\Omega\cap \Delta_{\bm\lambda}\big)\geq \dim(\Omega)+\dim(\Delta_{\bm\lambda})-(d+k).
	\end{align}
	Since $\varphi_1(\bm\lambda), \dots, \varphi_{k}(\bm\lambda)$ are matrices, it follows that $\Omega \cap \Delta_{\bm\lambda}$ is a finite set. So, $ \dim\big(\Omega\cap \Delta_{\bm\lambda}\big)=0$. It is easy to see that $\dim(\Delta_{\bm\lambda})=k$. Thus, by equation\eqref{dimEq}, we get 
	$$\dim(\Omega)\leq{d}.$$ 
For each $\bm z\in \bD^d$, there are finitely many points $\bm\lambda_{\bm z}$ in $\sigma(\bm\varphi(\bm z))$ such that $(\bm z,\bm\lambda_{\bm z})$ belongs to $\Omega$.	Thus there is an embedding of $\bD^d$ into $\Omega$, giving $\dim \Omega \ge d$.
Therefore, $\dim(\Omega)={d}$.

Finally, we show that $\Omega$ has no isolated point. 
    Suppose on the contrary that $(\bm\lambda,\bm\mu)\in\Omega$ is an isolated point. This means that 
    \[
    \mu=(\mu_1,\dots,\mu_k)\in \sigma(\varphi_1(\bm \lambda),\dots,\varphi_k(\bm \lambda)).
    \]
    Let $\{\bm \lambda_n\}\subset \mathbb D^d$ be such that $\bm \lambda_n\to \bm \lambda$. Then by Proposition~\ref{P:JointContinuity}, there exist joint eigenvalue $\bm \mu_{n_k}$  of $( \varphi_1(\bm \lambda_{n_k}),\dots, \varphi_k(\bm \lambda_{n_k}))$ such that $\bm\mu_{n_k}$ converges to $\bm \mu=(\mu_1,\dots,\mu_k)$. Thus $\Omega\ni(\bm\lambda_{n_k}, \bm\mu_{n_k}) \to (\bm \lambda, \bm\mu)$. This is the desired contradiction as $(\bm\lambda,\bm\mu)$ was assumed to be an isolated point of $\Omega$.
\end{proof}

\begin{remark}\label{Remark:variety}
Here we remark an alternative approach to establishing all the equalities in \eqref{desired identity}. Indeed, from the argument leading to the containment \eqref{HowThough}, we observed that the only non-trivial part of it is to prove the containment \eqref{HowThough}. In the proof given above, we used the definition of the Waelbroeck spectrum and the Cayley--Hamilton Theorem \ref{Thm:CayleyHamilton}. Here we note that one can prove the containment \eqref{HowThough} using the fact that the union $\Omega$ as in \eqref{TheUnion} is a variety. Indeed, since $\Omega$ as in \eqref{TheUnion} is a variety, by a result of Eisenbud and Evans \cite{EG}, there exist $d+k$ polynomials $\xi_j$ so that $\Omega=\cZ(\xi_1,\xi_2,\dots,\xi_{d+k})\cap\overline{\bD}^{d+k}$.
By the definition of the set $\Omega$ (see \eqref{TheUnion}), we have $$0=\xi_j\left(\sigma(\bm z I, \boldsymbol{\varphi}(\bm{z}))\right)=\sigma\left(\xi_j(\bm z I, \boldsymbol{\varphi}(\bm{z})\right)$$ for all $\bm z\in\overline{\bD}^d$ and for all $j$. Thus the matrices
	$\xi_j(\bm z I, \boldsymbol{\varphi}(\bm{z}))$ are nilpotent for all $\bm z\in\bD^d$. So we can safely say that for all $\bm z\in\bD^d$ and $j=1,2,\dots,k$,
    \begin{align}\label{E:tupleannihilating}
    \xi_j^N(\bm z I, \boldsymbol{\varphi}(\bm{z}))=0.
    \end{align}
    
	This implies that the polynomials $\xi_j^N$ annihilate the pair $(\bm M_{\bm z},\bm M_{\bm\varphi})$, and consequently we have \eqref{HowThough}. Let us note that the common thread in both approaches is to show that the $(d+k)$-tuple $(\bm M_{\bm z}, \bm M_{\bm\varphi})$ is annihilated by certain polynomials.
\end{remark}

With the help of of \cref{Poly1}, we obtain the Taylor spectrum of a $d$-tuple of commuting Toeplitz operators with matrix-valued rational symbols on $\bD^d$. 
\begin{thm}\label{Thm:M}
	Let $\varphi_1,\varphi_2, \dots, \varphi_{k}$ be commuting and contractive $N\times N$ matrix-valued rational functions with poles away from the closed polydisk $\overline{\bD}^d$. Then 
	\begin{align*}
		\sigma_{T}(M_{\varphi_1}, M_{\varphi_2}, \dots, M_{\varphi_k})=	\bigcup_{\bm z\in\overline{\mathbb{D}^d}}\sigma\big( \varphi_1(\bm z),\varphi_2(\bm z),\dots, \varphi_{k}(\bm z)\big).
	\end{align*}  
\end{thm}

\begin{proof}
Applying \cref{Poly1} to the commuting contractive operator-tuple 
$$
(M_{z_1}, \dots, M_{z_d}, M_{\varphi_1}, \dots, M_{\varphi_k}),
$$we obtain with $\bm z = (z_1,z_2,\dots,z_d)$,
$$\sigma_{T}(M_{z_1},\dots, M_{z_d}, M_{\varphi_1}, \dots, M_{\varphi_k}) 
= \bigcup_{\bm z \in \overline{\mathbb{D}^d}} \sigma\!\big(z_1I,\dots, z_dI, \varphi_1(\bm z), \dots, \varphi_d(\bm z)\big).$$ 
In the finite-dimensional case, the set of joint eigenvalues coincide with the Taylor spectrum, and the projection property holds for the Taylor joint spectrum, see \cite[Lemma 3.1]{Taylor}. The conclusion, therefore, follows once the above equality is projected onto the last $k$ coordinates. 
\end{proof}

\begin{remark}\label{R:w/oCoord}
The projection mapping property of the Taylor spectrum is the main driving force in the proof of \cref{Thm:M}. It is known that, unlike the Taylor spectrum, the Waelbroeck spectrum $\sigma_W(\cdot)$ does not enjoy the projection mapping theorem (see e.g. \cite{Cur}). It is not clear at this moment whether the Waelbroeck spectrum enjoys the same conclusion as in \cref{Thm:M}. However, let us note that if $\bm A=(A_1,A_2,\dots,A_k)$ is a commuting operator tuple whose Taylor joint spectrum is strictly contained in the Waelbroeck spectrum, then the constant functions
$\varphi_j(z)=A_j$ for each $j=1,2,\dots,k$ would yield an example of commuting contractive rational functions $\varphi_j$ so that 
\begin{align*}
	\sigma_{T}(M_{\varphi_1}, \dots, M_{\varphi_k})=	\bigcup_{\bm z\in\overline{\mathbb{D}^d}}\sigma\big( \varphi_1(\bm z),\dots, \varphi_{k}(\bm z)\big) \subsetneq \sigma_{W}(M_{\varphi_1}, \dots, M_{\varphi_k}).
\end{align*}
Indeed, suppose $\bm A$ is such a tuple. If $\varphi_j(z)\equiv A_j$, then $M_{\varphi_j} = I \otimes A_j$ and so
$$
\sigma_T(M_{\varphi_1},\dots,M_{\varphi_d})
    = \sigma_T(I\otimes A_1,\dots,I\otimes A_d)
    = \sigma_T(\bm A),
$$
because tensoring with the identity operator does not change the Taylor spectrum; see, for example, \cite{Mull}.
On the other hand, for any polynomial $p$,  
$$
p(I\otimes A_1,\dots,I\otimes A_d)
      = I\otimes p(A_1,\dots,A_d),
$$and since $\sigma(I\otimes T)=\sigma(T)$ for any bounded operator $T$, it follows that
\begin{align*}
\sigma_W(I\otimes A_1,\dots,I\otimes A_d)
    &= \{(\lambda_1,\dots, \lambda_d)\in\mathbb{C}^d : p(\lambda_1, \dots, \lambda_d)\in \sigma( A_1,\dots, A_d)\}
    =\sigma_W(\bm A).
\end{align*}
Therefore 
$$
\sigma_T(M_{\varphi_1},\dots,M_{\varphi_d})
    \subsetneq \sigma_W(M_{\varphi_1},\dots,M_{\varphi_d}).
$$However, the authors were unable to find a tuple whose Taylor spectrum is strictly contained in Waelbroeck spectrum. Although examples are known in the Banach space setting; see \cite[Section 4]{Taylor}.
\end{remark}

\subsection{Joint spectrum for commuting isometries}\label{SS:JointSpecCommIso}

Here we shall be concerned with commuting isometric tuples. Here we prove the results stated in items (\textbf{C}) and (\textbf{E}). 

Items (i) and (ii) in the result below appear in \cite[Theorem 3.2]{BKS-APDE} and the equivalence with the additional items (iii)-(v) were later realized in \cite[Theorems 3.3]{DasSauPAMS}. Here the finite dimensionality of the coefficient space, a result of Vesentini \cite[Corollay 4.5]{Ves} and the fact stated in \eqref{L:PureToeplitz} above are used. Every one-dimensional distinguished variety was characterised in \cite{BKS-APDE} as in the theorem below. Later in \cite{Pal}, it was shown that every distinguished variety is one-dimensional. This later development reflects in the converse part of the theorem below.

\begin{thm}\label{T:thm2}
Let $\varphi_1,\varphi_2,\dots,\varphi_d$ be BCL functions corresponding to orthogonal projections $P_1, P_2,\dots, P_d$ and unitary operators $U_1, U_2, \dots, U_d$ acting on a finite dimensional Hilbert space, i.e., $\varphi_j(z)=P_j^\perp U_j+ z P_j U_j$ for each $j$ (see Definition \ref{D:Commutator}). Then
\begin{align}\label{BDiskFoli}
\Omega:=\bigcup_{z \in \overline{\mathbb{D}}} \sigma(\varphi_1(z), \varphi_2(z), \ldots, \varphi_d(z))
\end{align}is a one dimensional algebraic variety in $\mathbb C^d$ with no isolated points.
Moreover, the following are equivalent:
\begin{itemize}
\item[(i)] $\Omega$ is a distinguished variety;
\item[(ii)]  For all $z$ in the open unit disk $\mathbb D$,
$ \nu(\varphi_j(z)) < 1  \text{ for all  } j=1,\dots, d$, where for a Hilbert space operator $A$, $\nu(A)$ denotes the numerical radius of $A$.
\item[(iii)] $\nu( P_j^\perp U_j) < 1  \text{ for each  } j=1,\dots, d$.
\item[(iv)] $(P_j^\perp U_j)^{*n}\to 0$ as $n\to\infty$ for each $j=1,2,\dots,d$;
\item[(v)] $M_{\varphi_j}^{*n}\to 0$ as $n\to\infty$ for each $j=1,2,\dots,d$.
\end{itemize}

Conversely, for any distinguished variety, there exist matrix-valued BCL functions $\varphi_1,\varphi_2,\dots,\varphi_d$ so that the variety is of the form \eqref{BDiskFoli}.
\end{thm}
The assertion that the algebraic variety \eqref{BDiskFoli} cannot have any isolated point is new and it is the consequence of the continuity of joint eigenvalues. Indeed, if $\bm\lambda\in\Omega$, then \eqref{varphis2} implies that $\bm\lambda\in\sigma(\varphi_1(\lambda),\varphi_2(\lambda),\dots,\varphi_d(\lambda)
)$ where $\lambda=\lambda_1\lambda_2\cdots\lambda_d$. We have $\bm\lambda_n:=(1-1/n)\bm\lambda\to\bm\lambda$ and so by Proposition \ref{P:JointContinuity}, $\bm\lambda$ must be a limit point of the following subset of $\Omega$:
$$
\{\sigma(\varphi_1(\lambda_n),\varphi_2(\lambda_n),\dots,\varphi_d(\lambda_n)):\lambda_n=(1-1/n)^d\lambda, \;n\geq1\}.
$$

Let us note from \cref{T:BCL} that if $(V_1,V_2,\dots,V_d)$ is a commuting isometric tuple so that $\dim\cD_{V^*}<\infty$ where $V=V_1V_2\cdots V_d$ is a pure isometry, then $(V_1,V_2,\dots,V_d)$ is jointly unitarily equivalent to $(M_{\varphi_1},M_{\varphi_2},\dots,M_{\varphi_d})$ where $\varphi_j$'s are matrix-valued BCL functions i.e., they are as in \eqref{varphis1} and \eqref{varphis2}. By  \cref{Thm:M}, the Taylor joint spectrum of $(M_{\varphi_1},M_{\varphi_2},\dots,M_{\varphi_d})$ is given by the union as in \eqref{BDiskFoli}. Thus in view of  \cref{Thm:M}, we can rephrase \cref{T:thm2} as in the theorem below. For the terminologies used in the result below, refer to Definition \ref{D:IsoDefect}.

\begin{thm}\label{T:JointSpecDisVar}
Let $(V_1,V_2,\dots,V_d)$ be a pure commuting isometric tuple of finite defect. Then $\sigma_T(V_1,V_2,\dots,V_d)$ is an algebraic variety. Moreover, $\sigma_T(V_1,V_2,\dots,V_d)$ is a distinguished variety if and only if each $V_j$ is a pure isometry.

Conversely, any distinguished variety is the Taylor joint spectrum of a pure commuting isometric tuple $(V_1,V_2,\dots,V_d)$ of finite defect with each $V_j$ a pure isometry.
\end{thm}
More is true for pure commuting isometric tuples of finite defect.
\begin{thm}\label{T:TheLastPiece}
Let $\bm V=(V_1,V_2,\dots,V_d)$ be a pure commuting isometric tuples of finite defect, and $\varphi_1,\varphi_2,\dots,\varphi_d$ be the corresponding matrix-valued BCL functions obtained by \cref{T:BCL}. Then
\begin{align}
\notag
\sigma_T(\bm V)&=\sigma_W(\bm V)=\bigcup_{z\in\overline{\bD}}\sigma(\varphi_1(z),\varphi_2(z),\dots,\varphi_d(z))\\
&= \{\bm\lambda\in\overline{\bD}^d: p_{\bm\alpha}(\bm\lambda)=0 \mbox{ for all }\bm\alpha\in\bC^d\},\label{TheLastPiece}
\end{align}
where for $\bm\alpha=(\alpha_1,\alpha_2,\dots,\alpha_d),$ $p_{\bm\alpha}$ is the polynomial given by
$$
p_{\bm\alpha}(z_1,z_2,\dots,z_d)=\det\bigg( \alpha_1(\varphi_1(z_1\cdots z_d)-z_1I)+\cdots+\alpha_d(\varphi_d(z_1\cdots z_d)-z_dI) \bigg).
$$
Moreover, $p_{\bm\alpha}(\bm V)=p_{\bm\alpha}(M_{\varphi_1},M_{\varphi_2},\dots, M_{\varphi_d})=0$ for every $\bm\alpha$ in $\bC^d$. Thus, $\bm V$ is a Cayley--Hamilton tuple.
\end{thm}
\begin{proof}
Since $(V_1,V_2,\dots,V_d)$ is jointly unitarily equivalent to $(M_{\varphi_1},M_{\varphi_2},\dots, M_{\varphi_d})$ for matrix-valued BCL functions $\varphi_j$, we can apply \cref{Poly1} to this case to get
\begin{align*}
\sigma_T(M_z,M_{\varphi_1},\dots,M_{\varphi_d})&
=\sigma_T(V_1V_2\cdots V_d,V_1,\dots,V_d)\\
&=\sigma_W(V_1V_2\cdots V_d,V_1,\dots,V_d)=\sigma_W(M_z,M_{\varphi_1},\dots,M_{\varphi_d})
\end{align*}
We now argue that for any commuting operator tuple $(A_1,A_2,\dots,A_d)$, one has
\begin{align*}
\mathbb P \bigg(\sigma_W(A_1A_2\cdots A_d,A_1,A_2,\dots,A_d)\bigg)= \sigma_W(A_1,A_2,\dots,A_d)
\end{align*}
where $\mathbb P$ is the projection onto the last $d$ coordinates. For notational convenience, we shall use $A=A_1A_2\cdots A_d$. The argument below actually yields
$$
\sigma_W(A,A_1,A_2,\dots,A_d)=
\{(\lambda_1\lambda_2\cdots\lambda_d,\lambda_1,\dots,\lambda_d):(\lambda_1,\dots,\lambda_d)\in\sigma_W(A_1,A_2,\dots,A_d)\}.
$$
Note that this together with the projection mapping property of the Taylor spectrum will yield the desired conclusion. Suppose $\bm\lambda=(\lambda_1,\lambda_2,\dots,\lambda_d)\in \sigma_W(A_1,A_2,\dots,A_d)$. With $\lambda=\lambda_1\lambda_2\cdots\lambda_d$, we show that $(\lambda,\lambda_1,\dots,\lambda_d)\in \sigma_W(A,A_1,\dots,A_d)$. Indeed, let $p$ be any polynomial in $d+1$ variables. Set the $d$-variable polynomial $q$ by
$$
q(z_1,z_2,\dots,z_d)=p(z_1z_2\cdots z_d,z_1,\dots,z_d).
$$Since $\bm\lambda\in \sigma_W(A_1,A_2,\dots,A_d)$, we have
$$
p(\lambda,\lambda_1,\dots,\lambda_d)=q(\lambda_1,\lambda_2,\dots,\lambda_d)\in\sigma( q(A_1,A_2,\dots,A_d))=\sigma (p(A,A_1,\dots,A_d)).
$$Conversely, suppose $(\lambda_0,\lambda_1,\dots,\lambda_d)\in \sigma_{W}(A, A_1,\dots, A_d)$. By considering polynomials in the last $d$ variables, we note that $(\lambda_1,\dots,\lambda_d)\in \sigma_{W}(A_1,\dots, A_d)$. Next, considering the $(d+1)$-variable polynomial $p(z_0,z_1,\dots,z_d)=z_0-z_1z_2\cdots z_d$, we have
$$
\lambda_0-\lambda_1\lambda_2\cdots\lambda_d=p(\lambda_0,\lambda
_1,\dots,\lambda_d)\in \sigma( p(A,A_1,\dots, A_d))=\sigma(0).
$$This shows that $\lambda_0=\lambda_1\lambda_1\cdots\lambda_d$. Therefore we have proved $\sigma_T(\bm V)=\sigma_W(\bm V)$.
We next show that 
\[
\Omega:=\bigcup_{z\in\overline{\bD}}\sigma(\varphi_1(z),\varphi_2(z),\dots,\varphi_d(z))\subseteq \{\bm\lambda\in\bC^d: p_{\bm\alpha}(\bm\lambda)=0 \mbox{ for all }\bm\alpha\in\bC^d\}. 
\]
Let $\bm \lambda=(\lambda_1,\dots,\lambda_d)\in \Omega $. Then there exists a point $z_0\in\mathbb D$ such that \[(\lambda_1,\dots,\lambda_d)\in \sigma(\varphi_1(z_0),\varphi_2(z_0),\dots,\varphi_d(z_0)).\] Since $\varphi_1(z_0)\cdots\varphi_d(z_0)=z_0 I$, it follows that $z_0=\lambda_1\cdots\lambda_d$. Therefore for any $\bm\alpha\in \mathbb C^d$, we have 
\[
p_{\bm\alpha}(\bm \lambda)=\text{det}\bigg( \alpha_1(\varphi_1(z_0)-\lambda_1I)+\cdots+\alpha_d(\varphi_d(z_0)-\lambda_dI) \bigg)=0.
\]
This proves the containment. For the other containment, suppose $\bm\lambda\in\overline{\bD}^d$ is such that $p_{\bm\alpha}(\bm\lambda)=0 \mbox{ for all }\bm\alpha\in\bC^d$. Then by Lemma~\ref{L:ZeroJoint},
\[
\bm\lambda\in \sigma(\varphi_1(\lambda),\varphi_2(\lambda),\dots,\varphi_d(\lambda))\quad\mbox{where } \lambda=\lambda_1\lambda_2\cdots\lambda_d\in \overline{\mathbb D}.
\]
This shows that
\begin{align*}
\{\bm\lambda\in \overline{\bD}^d : p_{\bm\alpha}(\bm\lambda)=0 \mbox{ for all }\bm\alpha\in\bC^d\}
\subset\bigcup_{z\in\overline{\bD}}\sigma(\varphi_1(z),\varphi_2(z),\dots,\varphi_d(z)).
\end{align*}
Finally, for the moreover part, let us first note that for any $\alpha\in\mathbb C^d$,
\[
\sigma(p_{\bm\alpha}(\bm V))= p_{\bm\alpha}(\sigma_W(\bm V))=\{0\}.
\]Since $p_{\bm\alpha}(\bm V)$ is subnormal, its norm equals its spectral radius by Stampli's theorem (see \cite{St}). The spectral radius of $p_{\bm\alpha}(\bm V)$ is zero from the above computation. 
\end{proof}

\begin{example}
Even in the simple case when $\psi_1,\psi_2,\dots,\psi_d$ are non-constant rational inner functions of the unit disk $\bD$, it is a daunting task to check whether or not the set
\begin{align}\label{AnalyticDisk}
\{(\psi_1(z),\psi_2(z),\dots,\psi_d(z)):z\in\overline{\bD}\}
\end{align}is an algebraic variety. Not only does \cref{T:TheLastPiece} imply that the set above is an algebraic variety but, a priori, it also gives the polynomials that determine it. Indeed, since $\psi_j$ are non-constant, $\psi_j(0)\in\bD$ and so each $M_{\psi_j}$ is a pure isometry. Moreover, each $M_{\psi_j}$ is of finite defect because $\psi_j$ are assumed to be rational and therefore by \cref{T:TheLastPiece} applied to the isometric tuple $(M_{\psi_1},M_{\psi_2},\dots,M_{\psi_d})$, the joint spectrum $\sigma_T(M_{\psi_1},M_{\psi_2},\dots,M_{\psi_d})$, which by \cref{Thm:M} is given by the set \eqref{AnalyticDisk}, is an algebraic variety. 
\end{example}

The results obtained in this subsection requires additional emphasis for the pair case, i.e., when $d=2$ in the results above. Indeed, when $d=2$, it can be checked from the conditions \eqref{varphis1} and \eqref{varphis2} that
$$
U_2=U_1^* \quad\mbox{and}\quad
P_2=U_1^*P_1^\perp U_1,
$$and therefore the BCL functions in this case are given by
\begin{align}\label{ShortHand2}
(\varphi_1(z),\varphi_2(z)):=(P^\perp U+zPU,U^*P+zU^*P^\perp).
\end{align}Thus, in contrast to the case of tuples of length strictly greater than $2$, one actually needs only one orthogonal projection and a unitary in Theorem 3.6. The case $d=2$ is also special from an algebraic geometry point of view. Indeed, a nontrivial algebraic variety in $\mathbb{C}^d$, $d \ge 3$, can be determined by at most $d$ polynomials (see \cite{EG}), whereas an algebraic variety in $\mathbb{C}^2$ can be determined by a single polynomial, see, for example, \cite{AL}. All this prompts the following summarization of \cref{T:thm2} and \cref{T:TheLastPiece} for the pair case in the Berger--Coburn--Lebow model form.
\begin{thm}\label{T:Pairjointspec}
Let $P$ be an orthogonal projection and $U$ be a unitary acting on a finite dimensional Hilbert space. Let $\varphi_1,\varphi_2$ be the functions as in \eqref{ShortHand2}. Then there exists a non-trivial polynomial $\xi\in\bC[z_1,z_2]$ such that
\begin{align}\label{PairJointSpec}
\sigma_T(M_{\varphi_1},M_{\varphi_2})=\sigma_W(M_{\varphi_1},M_{\varphi_2})= \bigcup_{z\in \overline{\bD}}\sigma(\varphi_1(z),\varphi_2(z))=\cZ(\xi)\cap\overline{\bD}^2.
\end{align}Moreover, the following are equivalent:
\begin{enumerate}
\item[(i)] $\sigma_T(M_{\varphi_1},M_{\varphi_2})$ (or equivalently, any other set in \eqref{PairJointSpec}) is a distinguished variety;
\item[(ii)] both $P^\perp U$ and $U^*P$ are pure.
\end{enumerate}
Conversely, any distinguished variety can be expressed in the form \eqref{PairJointSpec} for some $\varphi_1(z)$, $\varphi_2(z)$ as in \eqref{ShortHand2}.
\end{thm}
\begin{proof}
Here we only give a construction of the polynomial $\xi$ such that 
$$
\bigcup_{z\in \overline{\bD}}\sigma(\varphi_1(z),\varphi_2(z))= \cZ(\xi)\cap\overline{\bD}^2
$$ and leave the rest of the proof as it follows from \cref{T:thm2} and \cref{T:TheLastPiece}.  
In the case when $P=I$, then with $\sigma(U)=\{w_1,\dots, w_n\}$ we see that
\[\bigcup_{z\in \overline{\bD}}\sigma(\varphi_1(z),\varphi_2(z))=\bigcup_{i=1}^n\bigcup_{z\in \overline{\bD}}(zw_i, \overline{w_i})= \cZ\bigg(\prod_{i=1}^n(z_2-\overline{w_i})\bigg)\cap \overline\bD^2.
\]
We now treat the case $P\neq I$. Consider the polynomials
\[
p_i(z_1,z_2)=\textup{det}(\varphi_i(z_1z_2)-z_i I)\quad \mbox{for }i=1,2.
\]
Since $(\lambda_1,\lambda_2)\in\sigma(\varphi_1(z),\varphi_2(z))$ for some $z$ implies $z=\lambda_1\lambda_2$, it follows that 
\[
\bigcup_{z\in \overline{\bD}}\sigma(\varphi_1(z),\varphi_2(z))\subseteq \cZ(p_1)\cap \cZ(p_2)\cap \overline{\bD}^2.
\]
 On the other hand suppose $(0,0)\neq (\lambda_1,\lambda_2)\in \cZ(p_1)\cap \cZ(p_2)\cap \overline{\bD}^2$. Without any loss of generality we assume that $\lambda_1\neq 0$, as the other case can be treated similarly. Since $(\lambda_1,\lambda_2)\in \cZ(p_1)$, there exists a non-zero vector $x$ such that $\varphi_1(\lambda_1\lambda_2)x=\lambda_1x$. Then we have
\[
\lambda_1\lambda_2 x=\varphi_2(\lambda_1\lambda_2)\varphi_1(\lambda_1\lambda_2)x=\lambda_1 \varphi_2(\lambda_1\lambda_2)x.
\]
Therefore $\varphi_2(\lambda_1\lambda_2)x=\lambda_2x$, and consequently, $(\lambda_1,\lambda_2)\in \sigma (\varphi_1(\lambda_1\lambda_2), \varphi(\lambda_1\lambda_2))$.
Finally, by the continuity of the joint spectrum (Proposition \ref{P:JointContinuity}), no point of
 $\bigcup_{z\in \overline{\bD}}\sigma(\varphi_1(z),\varphi_2(z))$ can be isolated; in particular, the point $(0,0)$ must also belong to $\cZ(p_1)\cap \cZ(p_2)\cap \overline{\bD}^2$ if it belongs to $\bigcup_{z\in \overline{\bD}}\sigma(\varphi_1(z),\varphi_2(z))$. Thus we obtain 
\[
\bigcup_{z\in \overline{\bD}}\sigma(\varphi_1(z),\varphi_2(z))=\cZ(p_1)\cap \cZ(p_2)\cap \overline{\bD}^2=\cZ(\xi)\cap\overline{\bD}^2,
\]
where $\xi$ is the common factor of $p_1$ and $p_2$. The common factor $\xi$ must exist because otherwise the common zero set of $\{p_1,p_2\}$ would be a finite set, which contradicts the fact that it contains the infinite set $\bigcup_{z\in \overline{\bD}}\sigma(\varphi_1(z),\varphi_2(z))$.
\end{proof}

\subsection{Varieties as joint spectrum of commuting isometries}
Given a pure isometric tuple $\bm V$ of finite defect, \cref{T:TheLastPiece} says that the joint spectrum of $\bm V$ is an algebraic variety and provides a form of the polynomials that determine the variety. In this section, we restrict ourselves to the bi-variate case and ask what algebraic form must a polynomial $\xi$ in $\bC[z_1,z_2]$ have for it to give $\sigma(\bm V)=\cZ(\xi)\cap\overline{\bD}^2$? The answer is rewarding in that we obtain a two-way implication. We use \cref{T:JointSpecDisVar} and a decomposition result from \cite[Theorem 2.4]{BKS}. Both of these results are true for commuting isometric tuples of finite length, and we believe that a meticulous application of these results will yield a version of the result below for general tuples; we leave the details for future consideration.

 \begin{thm}\label{T:AnyPoly}
 	Let $\xi$ be a polynomial in two variables. There exists a pure commuting isometric pair $(V_1,V_2)$ with finite dimensional
defects such that
$$
\cZ(\xi)\cap\overline{\bD}^2=\sigma_T(V_1,V_2)
$$holds
if and only if there exist unimodular complex numbers $\alpha_i,\beta_j$ for $i=1,2,\dots,m$ and $j=1,2,\dots,n$ such that
\begin{align}\label{Equation:polynomial_structure}
\xi(z_1, z_2)
=
\prod_{i=1}^m (z_1-\alpha_i)\prod_{i=1}^n(z_2-\beta_i)\cdot \eta(z_1, z_2)\cdot \chi(z_1, z_2),
\end{align}
where $\chi$ is a polynomial such that $\cZ(\chi)\cap\overline{\bD}^2=\emptyset$. Moreover, any of the factors above may be absent.
 \end{thm}
 \begin{proof}
Let $(V_1,V_2)$ be a pair of commuting isometries such that $V_1V_2$ is pure and has finite-dimensional defect spaces, and suppose that
$\mathcal Z(\xi)\cap\overline{\mathbb D}^2=\sigma_T(V_1,V_2).$ By the model for
commuting isometries with finite defect \cite[Theorem 2.4]{BKS}, the pair $(V_1,V_2)$ is unitarily
equivalent to the direct sum
$$
(S_1,S_2)
\;\oplus\;
(M_z\otimes I,\, I\otimes U_1)
\;\oplus\;
(I\otimes U_2,\, M_z\otimes I)
\;\oplus\;
(\tau_1,\tau_2),
$$
where $(S_1,S_2)$ is a pair of commuting pure isometries, $U_1$ and $U_2$ are finite-dimensional unitary matrices, and $(\tau_1,\tau_2)$ is a pair
of commuting unitary operators. The readers are also referred to \cite[Section 3.2.2]{BS-Book} for a reduction of this and more general decomposition results from the model theory developed by Bercovici, Douglas and Foias in \cite{BDF-BiIso}. Since the product $V_1V_2$ is pure, the unitary--unitary
summand does not occur. Consequently,
$$\cZ(\xi)\cap\overline{\mathbb D}^2=
\sigma_T(S_1,S_2)
\;\cup\;
\sigma_T(M_z\otimes I,\, I\otimes U_1)
\;\cup\;
\sigma_T(I\otimes U_2,\, M_z\otimes I).
$$
Let $\{\beta_1,\ldots,\beta_n\}$ and $\{\alpha_1,\ldots,\alpha_m\}$ denote the
eigenvalues of $U_1$ and $U_2$, respectively. Using the fact that $
\sigma_T(M_z\otimes I,\, I\otimes T)
=
\overline{\mathbb D}\times\sigma(T),$
we obtain
$$
\sigma_T(M_z\otimes I,\, I\otimes U_1)
=
\bigcup_{i=1}^n \overline{\mathbb D}\times\{\beta_i\},
\quad \text{ and } \quad 
\sigma_T(I\otimes U_2,\, M_z\otimes I)
=
\bigcup_{i=1}^m \{\alpha_i\}\times\overline{\mathbb D}.
$$
Since these sets are contained in $\mathcal Z(\xi)\cap\overline{\mathbb D}^2$, it
follows that
\begin{align}\label{E:Fac}
\prod_{i=1}^n (z_2-\beta_i)
\quad\text{and}\quad
\prod_{i=1}^m (z_1-\alpha_i)
\end{align}
divide $\xi$. Next, note that the pair $(S_1,S_2)$ is a pair of commuting pure
isometries with finite defect. By \cref{T:JointSpecDisVar}, there exists a polynomial $\eta$ in two variables whose zero set $\cZ(\eta)$ is a distinguished variety such that
$$
\sigma_T(S_1,S_2)=\cZ(\eta)\cap \overline{\bD}^2.
$$
Therefore, all zeros of the polynomial $\xi$ in $\overline{\bD}^2$ are captured by \eqref{E:Fac} together with the polynomial $\eta$. Consequently, $\xi$ admits a factorization of the form
$$
\xi(z_1,z_2)
=
\prod_{i=1}^m (z_1-\alpha_i)
\prod_{i=1}^n (z_2-\beta_i)
\cdot
\eta(z_1,z_2)
\cdot
\chi(z_1,z_2),
$$
where $\chi$ does not vanish on $\overline{\bD}^2$.

 Conversely, suppose that $\xi$ is a polynomial of the form \eqref{Equation:polynomial_structure}.
 Then
$$
\mathcal{Z}(\xi) \cap \overline{\mathbb{D}}^2 = \Bigg(\bigcup_{i=1}^m \mathcal{Z}(z_1-\alpha_i) \,\cup \bigcup_{i=1}^n \mathcal{Z}(z_2-\beta_i) \,\cup \mathcal{Z}(\eta)\Bigg) \cap \overline{\mathbb{D}}^2.
$$
Since $\mathcal Z(\eta)$ determines a
distinguished variety, by \cref{T:JointSpecDisVar}, there exists a
pure commuting isometric pair $(V_1,V_2)$ on a Hilbert space $\mathcal H$ with
finite defect such that
$$
\sigma_T(V_1, V_2) = \mathcal{Z}(\eta) \cap \overline{\mathbb{D}}^2.
$$
Next, choose the diagonal unitary matrices $U_1=\operatorname{Diag}(\alpha_j)_{j=1}^m$ and $U_2=\operatorname{Diag}(\beta_j)_{j=1}^n$.
Consider the commuting isometric pair
$$
(W_1, W_2) =
\Bigg(
\begin{bmatrix}
V_1 & 0 & 0 \\
0 & M_z \otimes I & 0 \\
0 & 0 & I \otimes U_1
\end{bmatrix}, \;
\begin{bmatrix}
V_2 & 0 & 0 \\
0 & I \otimes U_2 & 0 \\
0 & 0 & M_z \otimes I
\end{bmatrix}
\Bigg).
$$
acting on the Hilbert space 
$$
\mathcal{H} \oplus \big(H^2\ \otimes \mathbb{C}^n\big) \oplus \big(H^2 \otimes \mathbb{C}^m\big).
$$
By construction, each $W_j$ has finite-dimensional defect spaces, and the product $W_1 W_2$ is pure. Finally, since
$
\sigma_T(M_z \otimes I, I \otimes T) = \overline{\mathbb{D}} \times \sigma(T),
$
we conclude
$$
\sigma_T(W_1, W_2)=\sigma_T(V_1,V_2)\cup\sigma_T(M_z\otimes I, I\times U_2)\cup\sigma_T(I\otimes U_1,M_z\otimes I)= \mathcal{Z}(\xi) \cap \overline{\mathbb{D}}^2,
$$
as desired.
\end{proof}

\subsection{Joint spectra and annihilating ideals}\label{S:Annihilation}

Here we prove the result stated in item (\textbf{D}). 
Recall from \eqref{TheSupp} in the introduction that if $T$ is a $C_0$ contraction with $\varphi_T$ as its minimal annihilating function, then 
$$
\sigma(T)=\operatorname{supp}(\varphi_T).
$$ We seek to obtain an analogous result in the present setting of Cayley--Hamilton tuples. This is done in the footsteps of the works done in \cite{CT, CK}.

We borrow the following notion of a support for the multi-variable setting from \cite{CT}.
\begin{definition}\label{D:support}
Let $\cA$ be any algebra of functions containing $\bC[\bm z]$, the algebra of polynomials in $d$ variables as a subalgebra. Let $J$ be any ideal of $\cA$. The \emph{support} of $J$ with respect to $\cA$, denoted by $\operatorname{supp}_\cA(J)$, is defined to be the following subset of $\mathbb{C}^d$:
$$
\operatorname{supp}_\cA(J)=\{
\bm\lambda=(\lambda_1,\lambda_2,\dots,\lambda_d):1 \notin J + (z_1 - \lambda_1) \cA + \cdots + (z_d - \lambda_d) \cA
\}.
$$
\end{definition} It follows from the definition that smaller the algebra, the larger the support, i.e., if $\cA$ and $\cB$ are two function algebras containing $\bC[\bm z]$ as a subalgebra such that $\cA\subset \cB$, then for any $J\subset \cA\subset \cB$
$$
\operatorname{supp}_\cB(J)\subset \operatorname{supp}_\cA(J).
$$Moreover, the support can vary largely, viz., $\operatorname{supp}_{\bC[\bm z]}(\{0\})=\bC^d$ while $\operatorname{supp}_{\bC[\bm z]}(\bC[\bm z])=\emptyset$. However, the support must be bounded if we agree to include $\operatorname{Rat}(\bD^d)$, the rational functions with poles outside $\overline{\mathbb{D}}^d$, in the function algebra $\cA$.
\begin{lemma}\label{Supp;L1}
	If $\cA$ contains $\operatorname{Rat}(\bD^d)$, the algebra of rational functions with poles off $\overline{\bD}^d$, then $
	\operatorname{supp}_\cA(J) \subset \overline{\mathbb{D}}^d$ for any $J\subset \cA$, even when $J=\{0\}$ or empty.
\end{lemma}
\begin{proof}
 Let $\bm\lambda=(\lambda_1,\dots,\lambda_d)$ be outside $\overline{\mathbb{D}}^d$. Suppose without loss of generality $|\lambda_1|>1$. Then with the rational function
	$
	g(z_1,\dots,z_d) = (z_1-\lambda_1)^{-1},
	$
we have $(z_1-\lambda_1)g\equiv 1$ on $\mathbb{D}^d$, and so
	$1\in J+(z_1-\lambda_1)\cA$.
  This means $\bm\lambda\notin \operatorname{supp}_\cA(J)$.
\end{proof}
\begin{definition}
Given a commuting operator tuple $\bm T=(T_1,T_2,\dots,T_d)$ acting on a Hilbert space $\cH$ and an algebra $\cA$ of functions on $\sigma_T(\bm T)$ containing $\bC[\bm z]$ as a subalgebra, we say that $\bm T$ possesses an \textit{$\cA$-functional calculus} if there exists an algebra homomorphism $\pi:\cA\to \cB(\cH)$ that extends the polynomial functional calculus. We shall use the notation $$
\operatorname{Ann}_\cA(\bm T):=\operatorname{ker}\pi.
$$
\end{definition}Let us emphasize here that no topological assumption is made on the algebra homomorphism $\pi$ in the above definition. So, while $\operatorname{Ann}_\cA(\bm T)$ is always an ideal, no topological behaviour of it is in discussion. Let us also note that by the work of Taylor \cite{Taylor}, every commuting operator tuple $\bm T$ possesses a $\operatorname{Hol}(\sigma_T(\bm T))$-functional calculus, where $\operatorname{Hol}(\sigma_T(\bm T))$ is the algebra of functions that are holomorphic in a neighbourhood of $\sigma_T(\bm T)$.

\begin{lemma}\label{L:GenLem}
Let $\cA$ be a function algebra containing $\bC[\bm z]$ as a subalgebra, and let $\bm T=(T_1,T_2,\dots,T_d)$ be any commuting operator tuple possessing an $\cA$-functional calculus and $\sigma(\cdot)$ be any notion of joint spectrum so that for every $f\in\cA$,
$$
f(\sigma(\bm T))=\sigma(f(\bm T))\quad\mbox{and}\quad
\sigma_T(\bm T)\subset \sigma(\bm T).
$$ Then we have 
\begin{align}\label{IntersectSupp'}
\sigma_T(\bm T)\subset \sigma(\bm T)\subset \cZ(\operatorname{Ann}_\cA(\bm T)) \subset \operatorname{supp}(\operatorname{Ann}_\cA(\bm T)).
\end{align}
\end{lemma}
\begin{proof}
By the first property of $\sigma(\cdot)$ we have for every $f\in \operatorname{Ann}_\cA(\bm T)$, $\{0\}=\sigma(f(\bm T))=f(\sigma(\bm T))$ proving the second containment. The last containment follows from the definition of the support.
\end{proof}

An important remark is in order.
\begin{remark}\label{R:ImpRemark}
Suppose $\cA$ is a function algebra with $\bC[\bm z]$ as a subalgebra, and
$\bm T = (T_1, T_2, \dots , T_d)$ is any commuting operator tuple possessing an $\cA$-functional calculus. Then one always has $\sigma_T(\bm T)\subset \operatorname{supp}(\operatorname{Ann}_{\cA}(\bm T))$. Indeed, if $\bm\lambda$ is not in $\operatorname{supp}(\operatorname{Ann}_{\cA}(\bm T))$, then there exists $g\in \operatorname{Ann}_{\cA}(\bm T)$ and $f_j\in \cA$ so that 
$$
1= g + (z_1-\lambda_1)f_1 + (z_2-\lambda_2)f_2 +\cdots + (z_d-\lambda_d)f_d.
$$Applying the functional calculus, we see that 
$$
I = (T_1-\lambda_1I)f_1(\bm T) + (T_2-\lambda_2I)f_2(\bm T) + \cdots + (T_d-\lambda_dI)f_d(\bm T).
$$This implies that $\bm\lambda\notin\sigma_T(\bm T)$ by \cite[Proposition 25.3]{Mull}.
Therefore if one desires only the inclusion $\sigma_T(\bm T)\subset \operatorname{supp}(\operatorname{Ann}_{\cA}(\bm T))$, then one does not need the additional hypothesis made in Lemma \ref{L:GenLem} on the functional calculus.
\end{remark}
\begin{thm}\label{theorem:SS}
Let $\bm T=(T_1, \cdots, T_d)$ be a Cayley--Hamilton tuple and $\sigma(\cdot)$ be any notion of joint spectrum so that for every polynomial $p$,
\begin{align}\label{AdmissSpec}
p(\sigma(\bm T))=\sigma(p(\bm T))\quad\mbox{and}\quad
\sigma_T(\bm T)\subset \sigma(\bm T).
\end{align}
Then we have
$$
\cZ(\operatorname{Ann}_{\bC[\bm z]}(\bm T))\cap\overline{\bD}^d=\sigma_T(\bm T)=\sigma(\bm T)=\operatorname{supp}(\operatorname{Ann}_{\bC[\bm z]}(\bm T))\cap\overline{\bD}^d.
$$
\end{thm}
\begin{proof}
Half of the work is done in Lemma \ref{L:GenLem}. It remains to prove 
$$
\operatorname{supp}(\operatorname{Ann}_{\bC[\bm z]}(\bm T))\cap\overline{\bD}^d \subset \sigma_T(\bm T).
$$ Since $\bm T$ is a Cayley--Hamilton tuple, by a result of Eisenbud and Evans \cite{EG}, there exist polynomials $\xi_1,\dots,\xi_d$ such that
\begin{enumerate}
\item[(i)] $\xi_j(\bm T)=0$ for each $j=1,\dots,d$; and
\item[(ii)] $\sigma_T(\bm T)=\bigcap_{j=1}^d \mathcal Z(\xi_j)\cap \overline{\mathbb D}^{,d}=\cZ(\operatorname{Ann}_{\bC[\bm z]}(\bm T))\cap\overline{\bD}^d$.
\end{enumerate}
Now suppose
$\bm\lambda=(\lambda_1,\dots,\lambda_{d})\in\overline{\bD}^d$ is not in $\sigma_T(\bm T)$. 
By (ii), there exists $1\leq\ell\leq d$ such that $
\xi_\ell(\bm\lambda) \neq 0$ but by (i), $\xi_l(\bm T)=0$. There exist polynomials $\eta_1,\dots,\eta_{d}$ such that 
$$
\xi_l(\bm z) - \xi_l(\bm\lambda) 
= \sum_{j=1}^{d} (z_j-\lambda_j)\eta_j(\bm z),
$$which, on account of $\xi_l(\bm\lambda) $ being non-zero, can be rearranged to obtain
$$
1= \frac{1}{\xi_l(\bm\lambda)}\zeta_\ell(\bm z) + \sum_{j=1}^{d} (z_j-\lambda_j)\frac{1}{\xi_l(\bm\lambda)}\eta_j(\bm z).
$$
This means $\bm\lambda$ does not lie in $\operatorname{supp}(\operatorname{Ann}_{\bC[\bm z]}(\bm{T}))$. This proves the theorem.
\end{proof}
As discussed in \S \ref{SS:JS}, the Waelbroeck joint spectrum has the properties \eqref{AdmissSpec} and thus the statement in item (\textbf{D}) in the Introduction is proved. We end with a remark. 
\begin{remark}
It is worth mentioning that although a $H^\infty$-functional calculus is not available for a general contractive operator tuple, there is a meaningful $H^\infty$-functional calculus for certain tuples of Toeplitz operators. Indeed, if $\varphi_j$ are analytic functions on $\bD^d$ with the additional property that $\sigma(\varphi_1(\bm z),\varphi_2(\bm z),\dots,\varphi_k(\bm z))\subset\bD^k$ for every $\bm z\in\bD^d$, then the Toeplitz tuple $\bm M_{\bm\varphi}=(M_{\varphi_1},M_{\varphi_2},\dots,M_{\varphi_k})$ admits a $H^\infty(\bD^k)$-functional calculus by defining
$$
f(M_{\varphi_1},M_{\varphi_2},\dots,M_{\varphi_k}) := M_{f(\varphi_1,\varphi_2,\dots,\varphi_k)}
$$
for every $f$ in $H^{\infty}(\bD^d)$. Here, the symbol 
$$
\bD^d \ni \bm z\mapsto f(\varphi_1(\bm z),\varphi_2(\bm z),\dots,\varphi_k(\bm z))
$$makes use of the usual Taylor functional calculus, which is possible because of the added spectrum assumption on the functions $\varphi_j$. This functional calculus has a continuity property in that if $f_n\to f$ uniformly over compact subsets of $\bD^k$, then by the continuity of the Taylor functional calculus, it follows that $f_n(\bm M_{\bm\varphi})\to f(\bm M_{\bm \varphi})$ in the operator norm. The Cayley--Hamilton Toeplitz tuples described in \cref{Poly1} and \cref{T:TheLastPiece} are Toeplitz tuples with rational matrix-valued functions. If the symbols in these theorems satisfy the added hypothesis on the joint eigenvalues as above, the annihilating ideal $\operatorname{Ann}(\bm T)$ could be taken in the $H^\infty$-algebra instead of the polynomial algebra. 
Consequently, by Remark \ref{R:ImpRemark} and \cref{theorem:SS}, the Taylor joint spectrum of the Cayley--Hamilton Toeplitz tuples described in \cref{Poly1} and \cref{T:TheLastPiece} coincides with the support of the annihilating ideal in $H^{\infty}$ provided that the symbols under considerations satisfy the additional assumption on its joint spectrum as above.
\end{remark}

\section{Further Problems} 

There emerged a number of open problems in the course of development of the paper. Let us reiterate that it is quite rare for a tuple to have its Taylor joint spectrum as an algebraic variety. One stumbles upon non-examples more than an example. Indeed, take any Blaschke function $B$ with its zero set being an infinite set, say $\{\alpha_n:n\geq1\}$. Let $\cQ:=H^2\ominus B\cdot H^2$ and $\psi$ be any non-constant rational inner function. Denote 
$$
(T_1,T_2):=(P_{\cQ}M_z|_{\cQ},P_{\cQ}M_\psi|_{\cQ}).
$$ It follows from \eqref{TheSupp} that $\sigma(T_1)\cap\bD=\{\alpha_n:n\geq1\}$. From \eqref{Compress} it follows that $\psi(T_1)=T_2$. By the mapping theorem and since $\psi$ is inner it follows that
$$
\sigma(T_2)\cap\bD=\sigma(P_{\cQ}M_\psi|_{\cQ})\cap\bD = \{\psi(\alpha_n):n\geq1\}.
$$
By the inclusion $\sigma_T(T_1,T_2)\subset \sigma(T_1)\times\sigma(T_2)$ we therefore have
$$
\{(\alpha_n,\psi(\alpha_n)):n\geq1\}\subset\sigma_T(P_{\cQ}M_z|_{\cQ},P_{\cQ}M_\psi|_{\cQ})\cap \bD^2 \subset \{\alpha_n\}_{n\geq1}\times\{\psi(\alpha_n)\}_{n\geq1}.
$$So, $\sigma_T(P_{\cQ}M_z|_{\cQ},P_{\cQ}M_\psi|_{\cQ})\cap \bD^2$ is an infinite discrete set, and no such set can be an algebraic variety; although it is contained in the variety determined by the polynomial in the numerator of $z_2-\psi(z_1)$ when $\psi$ is written as the quotient of two polynomials.

In continuation of the above discussion, let us remark that while the spectrum being an algebraic variety need not imply that the underlined tuple is a Cayley–Hamilton tuple (see Example \ref{E:Volt}), there are certain classes of operator tuples for which the spectrum being an algebraic variety is enough. Consider for example, the classes (1) contractive tuples of subnormal operators;
and (2) contractive tuples of analytic Toeplitz operators with matrix-valued symbols.
Indeed, suppose $\bm T$ is in one of the classes above. If $\sigma_{T}(\bm T)=\cZ(\cI)\cap\overline{\bD}^d$ for some collection of polynomials $\cI$. Then by the polynomial mapping theorem, for any $p\in \cI$,  
\begin{align}\label{PolyMapp}
\sigma(p(\bm T))=p(\sigma_{T}(\bm T))=\{0\}.
\end{align}
Now if $\bm T$ is subnormal, then so is $p(\bm T)$. By Stampli's theorem (see \cite{St}), the spectral radius of $p(\bm T)$ is the same as its norm. Thus $p(\bm T)=0$. To see (2), let $\bm T=(M_{\varphi_j})_{j=1}^d$ be a commuting tuple of analytic Toeplitz operators acting on $H^2_{\bD^d}(\bC^N)$. Thus the symbols are assumed to be $N\times N$ matrix-valued but we do not require that the symbols be continuous up to the boundary. Then $p(\bm T)=M_{p(\varphi_1,\dots,\varphi_d)}$ is also a Toeplitz operator for every polynomial $p$. The fact \eqref{PolyMapp} then forces $\sigma(p(\varphi_1(z),\dots,\varphi_d(z)))=0$ for every $z\in\bD$ and for every $p\in\cI$. Since the symbols are $N\times N$ matrix-valued, it is guaranteed that $p^N(\bm T)=0$. Because this holds for all $p\in \cI$, we conclude that $\bm T$ is a Cayley-Hamilton tuple with $\cI^N:=\{p^N:p\in\cI\}$ as the underlined collection of polynomials.

A characterization for general Cayley--Hamilton tuples remains in limbo. Therefore the problem of utmost interest is to
\begin{problem}
Characterize the Cayley--Hamilton tuples.
\end{problem}
A parallel problem of interest is to 
\begin{problem}
Characterize commuting contractive operator tuples whose Taylor joint spectrum is an algebraic variety (see Definition \ref{D:AlgVar}).
\end{problem}
The Taylor spectrum for commuting Toeplitz operators with general operator-valued symbols remains unknown. \cref{Thm:M} computes this when the symbols are rational and matrix-valued. For a scalar-valued bounded analytic function $\varphi$ it is well known that $\sigma(T_\varphi)=\overline{\varphi(\bD)}$. It is therefore expected that \cref{Thm:M} is true in more generality.
\begin{problem}
For which commuting operator-valued analytic symbols $\varphi_j$ do we have
\begin{align*}
\sigma_T(M_{\varphi_1},M_{\varphi_2},\dots,M_{\varphi_d})=
\overline{\bigcup_{z\in\bD}\sigma(\varphi_1(z),\varphi_2(z),\dots,\varphi_d(z))}?
\end{align*}
\end{problem}

\smallskip
\noindent
\textbf{Acknowledgment.} The authors Kumar and Sau and grateful to the Indian Institute of Technology Bombay for the generous hospitality and supportive workplace they were provided during the course of the present work. The second author profusely thanks Professor T. Bhattacharyya for the financial support he received from Professor Bhattacharyya's J. C. Bose Fellowship (JCB/2021/000041) of SERB.

\end{document}